\documentclass[12pt]{amsart}
\usepackage{amssymb, latexsym, mathrsfs, color, tikz, multirow,mathtools,xcolor,enumerate,caption, subcaption, graphicx,xcolor,comment,multicol}
\usepackage{graphics}
\graphicspath{ {./Diagrams/} }

\input{xy}
\xyoption{all}

\usepackage[colorlinks=true, pdfstartview=FitV,
 linkcolor=blue,citecolor=blue,urlcolor=blue]{hyperref}
\usetikzlibrary{shapes,snakes}

\usepackage{ytableau} 

\numberwithin{equation}{section}
\theoremstyle{plain}
\newtheorem{theorem}[equation]{Theorem}
\newtheorem{proposition}[equation]{Proposition}
\newtheorem{corollary}[equation]{Corollary}
\newtheorem{lemma}[equation]{Lemma}

\newtheorem {question}[equation]{Question}

\newtheorem{innercustomthm}{Theorem}
\newenvironment{customthm}[1]
  {\renewcommand\theinnercustomthm{#1}\innercustomthm}
  {\endinnercustomthm}

\newtheorem{innercustomcor}{Corollary}
\newenvironment{customcor}[1]
  {\renewcommand\theinnercustomcor{#1}\innercustomcor}
  {\endinnercustomthm}
  
\newtheorem{innercustomlem}{Lemma}
\newenvironment{customlem}[1]
  {\renewcommand\theinnercustomlem{#1}\innercustomlem}
  {\endinnercustomthm}
\newtheorem{innercustomprop}{Proposition}
\newenvironment{customprop}[1]
  {\renewcommand\theinnercustomprop{#1}\innercustomprop}
  {\endinnercustomprop}

\theoremstyle{definition}
\newtheorem{definition}[equation]{Definition}
\newtheorem{example}[equation]{Example}
\newtheorem{remark}[equation]{Remark}

\def\N{\mathbb N}

\keywords{dice relabeling, cyclotomic polynomials, generating functions}


\title{Revisiting Dice Relabeling Using Cyclotomic Polynomials}

\author{Yikai Chao, Josh Gabel, Carlye Larson, George D.  Nasr}

\address{Department of Mathematics, Augustana University} \email{george.nasr@augie.edu}

\begin{document}
\maketitle

\begin{abstract}
We continue the exploration of a question of dice relabeling posed by Gallian and Rusin: Given $n$ dice, each labeled $1$ through $m$, how many ways are there to relabel the dice without changing the frequencies of the possible sums? We answer this question in the case where $n=2$ and $m$ is a product of three prime numbers. We also explore more general questions. We find a method for decomposing two $m$-sided dice into two dice of different sizes and give some preliminary results on relabeling two dice of different sizes. Finally, we refine a result of the aforementioned authors in the case where $m$ is a prime power.
\end{abstract}

\section{Introduction and Summary of Results}

George Sicherman posed and solved the following question.

\begin{question}
 How many ways can one label two six-sided dice so that the frequency of all possible sums remain the same as if they were both labeled $1$ through $6$?
\end{question}

Sicherman found that the answer was two. Either one can use the usual labeling on both dice (often called the ``standard" solution), or one can label one dice $1,2,2,3,3,4$ and the other dice $1,3,4,5,6,8$. This result was discussed further and reported by Martin Gardner \cite{gardner}. Inspired by this work, Broline explored this question for an arbitrary number of platonic solids \cite{b}. Gallian and Rusin addressed the more general question \cite{gc}.

\begin{question}\label{qu:gen}
  Given $n$ dice, each with labels $1$ through $m$, how many way can these dice be relabeled without altering the frequencies of the sums? 
\end{question}
 By encoding the data of the frequencies and the labels on the dice as a generating function, Broline, Gallian, and Rusin observed that one could factor the generating function encoding the frequencies using cyclotomic polynomials (a method we go into further detail on later). Using this technique, Gallian and Rusin were able to demonstrate that for any number of dice, there are three possible dice that could be used to answer Question \ref{qu:gen} if $m$, the number of sides, is a product of two (not necessarily distinct) prime numbers \cite[Theorem 2]{gc}. They additionally have results for when $m$ is a prime power and many other related questions to relabeling dice. 

Many different types of results involving dice relabeling followed. In \cite{sm}, they enumerate the frequency of a particular sum given $n$ $m$-sided dice. In \cite{fs}, they characterize the numbers that can be realized as the sums of relabeled six-sided dice. Other papers explored changing the probabilities of the sums from the usual one given by $n$ $m$-sided dice. For instance, authors of \cite{bmrs,lr,bs,m} explored different questions assuming ``equally likely sums", that is, all sums are equally likely. In \cite{rss}, they consider ``Pythagorean dice" which provide an alternative probability distribution on the possible sums. 

However, there remains many generalizations of Gallian's and Rusin's results involving the case where we use the original probabilities. Indeed, at the end of their paper, Gallian and Rusin leave the readers with two different further explorations of their ideas.
\begin{question}\label{qu:three}
How many relabeling are there in the case where $m=p^2q$ or $m=pqr$ (where $p,q,r$ are distinct)? 
\end{question}
\begin{question} \label{qu:dif_size}
Can one find dice, not necessarily with the same number of sides, matching the frequencies of $n$ $m$-sided dice? 
\end{question}

In this paper, following the techniques of Gallian and Rusin, we address both these questions in the case where the number of dice, $n$, is two. Our paper is organized as follows. In Section \ref{sec:term}, we go over terminology for this paper. In Section \ref{sec:cyc}, we go over how we use generating functions and cyclotomic polynomials to reframe the above questions. In this section, we also include cyclotomic polynomial identities, proving anything we did not readily find in the literature. In section \ref{sec:p2q} and section \ref{sec:pqr}, we address Question \ref{qu:three}, separately dealing with the two cases. In particular, we show the following.

\begin{customthm}{\ref{thm:p2q}}
Let $p$ and $q$ be distinct prime numbers. There are $8$ pairs of dice of size $p^2q$ whose frequencies of possible sums agrees with that of two $p^2q$ dice labeled $1$ through $p^2q$.
\end{customthm}

\begin{customthm}{\ref{thm:pqr}}
Let $p$, $q$, and $r$ be distinct prime numbers. There are $13$ pairs of dice of size $pqr$ whose frequencies of possible sums agrees with that of two $pqr$ dice labeled $1$ through $pqr$.
\end{customthm}
\noindent As an application of our results, we provide explicit lists of all possible labelings for the smallest cases of both of these results in their respective section. From here, we explore types of generalizations of the question explored in \cite{gc}. In section \ref{sec:new_sizes}, we give our first approach to answering Question \ref{qu:dif_size} by providing a way to take two standard $m$-sided dice, and achieve the same frequencies with two dice of sizes $a$ and $b$ where $m=ab$. Specifically, we demonstrate the following result, which is a Corollary to a result in Section \ref{sec:new_sizes}.
\begin{customcor}{\ref{cor:carlye}}
Let $m=ab$ with $a$ and $b$ non-negative integers. 
Consider a dice size $ab^2$ who labels come from $\{1,2,\dots, 2m-a\}$ in the following way:
\begin{enumerate}
\item the numbers $(i-1)a+1,(i-1)a+2,\dots, ia, 2m-(i+1)a+1, 2m-(i+1)a+2,\dots, 2m-ai$ each appear $i$ times on the dice for $1\leq i\leq b-1$; and 
\item the numbers $m-a+1,m-a+2,\dots, m$ appear $b$ times on the dice.
\end{enumerate}
This dice, along with a $a$-sided dice with labels $1$ through $a$, has the same frequencies of sums as two $m$-sided dice labeled $1$ through $m$.
\end{customcor}
\noindent We will also discuss a surprising combinatorial connection to triangular numbers. In section \ref{sec:dif_sizes}, we explore the following question.
\begin{question}\label{qu:dif_dice_size}
Given an $m_1$-sided dice, labeled 1 through $m_1$, and an $m_2$-sided dice, labeled $1$ through $m_2$, how many ways can one relabel both dice while not changing the frequencies of the sum?
\end{question} 

\noindent We report preliminary results on this direction, namely, the following. 
 \begin{customprop}{\ref{lem:prime}}
The answer to Question \ref{qu:dif_dice_size} is none when $m_1$ and $m_2$ are distinct prime numbers.
 \end{customprop}
 \begin{customlem}{\ref{lem:prime_power}}
Suppose we have a $p$-sided dice with labels $1$ through $p$ and a $p^k$-sided dice with labels $1$ through $p^k$. There are $k$ ways to relabel these dice without changing the frequencies of the possible sums.
 \end{customlem}
  \noindent Finally, in section \ref{sec:prime_power}, we refine the case of prime power sized dice addressed in \cite[Theorem 10]{gc}. 

\subsection*{Acknowledgments}

We thank Shriram Maiya (Augustana University) for motivating the idea of writing rational functions as products of series and Lemma \ref{lem:shriram}. We additionally thank Augustana University for providing us with the Froiland Research Grant to make this project possible. We thank David Rusin for communicating a counter example to a conjecture in the original version of this paper. We finally thank the anonymous referee who took the time to review the paper and provide us with useful feedback. 

\section{Terminology}\label{sec:term}

We follow the primary terminology set up by \cite{gc}, namely: 
\begin{definition}\leavevmode
\begin{itemize}
\item The dice labeled $1$ through $m$ is called a \textit{standard dice}. 
\item A dice with $m$ sides has \textit{size $m$}.
\item Given a set of $n$ dice which have the same frequencies of sums of $n$ standard $m$-sided dice, any one of these dice is called a \textit{solution}.
\end{itemize}
\end{definition}

Thus, Gallian and Rusin showed that when $m$ is a product of two (not necessarily distinct) prime numbers, there are three solutions. When $n=2$, this gives rise to two possible pairs: two standard dice and the two other dice guaranteed by their result. 

\begin{remark}
Technically, when \cite{gc} defines ``solution", they require the corresponding dice have size $m$. We remove this condition as our results in Sections \ref{sec:new_sizes} and \ref{sec:dif_sizes} allow for dice to have different sizes.
\end{remark}

\section{Using Generating Functions and Cyclotomic Polynomials} \label{sec:cyc}

Question \ref{qu:gen} can be reframed in the following way: How many collections of polynomials $P_1,P_2,\dots, P_n$, all with non-negative integer coefficients, are there so that  $P_i(1)=m$ for all $i$ and
\[P_1P_2\cdots P_n=\left( \sum_{i=1}^m x^i \right)^n=x^n\left( {x^m-1\over x-1}\right)^n?\]

Here, each polynomial $P_i$ is the generating function of a given dice. That is, $[x^j]P_i$, the $j$th coefficient of $P_i$, is the number of sides of the $i$th dice labeled $j$. Thus, $\displaystyle \sum_{i=1}^m x^i$ is the generating function for the standard dice of size $m$. Thus, if we denote the product of the $P_i$'s as $F(x)$, then $[x^j]F$ is the number of ways which a sum of $j$ can be achieved from the $n$ dice corresponding to the polynomials $P_1,P_2,\dots, P_n$. We refer to $F$ as the \textit{frequency polynomial}. Throughout this paper, we will see different formulations of $F$ depending on the size of the dice involved. Regardless of the explicit formulation of $F$, observe this gives a bijection between collections of dice answering our question and factorizations of $F$ with $P_i(1)=m$ and $[x^j]P_i\geq 0$ for all $i$ and $j$. To this end, we also refer to each $P_i$ as \textit{solution}, keeping in mind we really are referring to the dice that $P_i$ represents.

Observe that solving this question relies on understanding the factorizations of $x^m-1$. The factors of this polynomial is completely understood, as its irreducible factors are the \textit{cyclotomic polynomials}. Results on these polynomials are well understood and we will survey some of what is known here, starting with the definitions.

\begin{definition}
Let $m$ be a non-negative integer. 
\begin{itemize}
\item A (complex-valued) solution to $x^m=1$ is called an \textit{$m$th root of unity}.
\item A \textit{primitive $m$th root of unity} is an $m$th root of unity which is not a $j$th root of unity for any $j<m$.
\item The \textit{cyclotomic polynomial} $\phi_m(x)$ is the polynomial whose roots are the primitive $m$th roots of unity.
\end{itemize}
\end{definition}

The following known identities or readily computed by identities that we will use throughout our paper. In what follows, $\mu(n)$ is the \emph{M\"obius function}, which is $1$ when $n$ is square-free with an even number of prime numbers, $-1$ when $n$ is square-free with an odd number of prime numbers, and $0$ otherwise. Additionally, throughout, we assume $p,q,r$ are a prime numbers. 

\begin{align}
x^n-1 &=\prod_{d|n}\phi_d(x)\label{eq:prod_cyc}\\
\phi_n(x) &=\prod_{d|n}(x^d-1)^{\mu(n/d)}\label{eq:prod_roots}\\
\phi_p(x)&=\sum_{i=0}^{p-1}x^i={x^p-1\over x-1}\label{eq:cyc_prime}\\
\phi_{p^km}&=\phi_{pm}(x^{p^{k-1}}), \text{where $m$ is relatively prime to $p$.}\label{eq:cyc_prime_power_mult}\\
\phi_m\phi_{pm} &=\phi_m(x^p), \text{where prime $p$ does not divide $m$.}\label{eq:cyc_gen}\\
\phi_{n}(1)&=\begin{cases}p & n=p^k\\1 & \text{otherwise} \end{cases}\label{eq:eval_prime_power}\\
\phi_{p^kq}(x) &= \frac{(x^{p^{k-1}}-1)(x^{p^kq}-1)} {(x^{p^k}-1)(x^{p^{k-1}q}-1)}\\
\phi_{pqr}(x)&=\frac{(x^p-1)(x^q-1)(x^r-1)(x^{p q r}-1)}{(x-1)(x^{p q}-1)(x^{pr}-1)(x^{q r}-1)}
\end{align}

\begin{remark}
One should note that the factors of the form $x^i-1$ could all be written as $1-x^i$ and the results, as far as this paper is concerned, will remain unchanged. In future sections, we will often write rational functions in both ways, depending on whatever is convenient for our proofs.
\end{remark}

We will state and prove some additional identities which we will use but could not find referenced anywhere. 

\begin{lemma}
If $p$ and $q$ are prime numbers, then for all integers $k\geq 0$,
\[ \prod_{i=0}^k \phi_{p^iq}(x)=\phi_q(x^{p^k})\]\label{lem:product_piq}
\end{lemma}

\begin{proof}
We proceed by induction on $k$. Observe that the statement holds for $k=0$. If $k>0$, then by induction, we have
\begin{align}
\left(\prod_{i=0}^{k-1} \phi_{p^{i}q}(x)\right)\phi_{p^kq}(x)=\phi_q(x^{p^{k-1}})\phi_{p^kq}(x) \label{eq:piq}
\end{align}

Using identities \eqref{eq:cyc_prime_power_mult} and \eqref{eq:cyc_gen}, observe that 
\[\phi_{p^kq}(x)=\phi_{pq}(x^{p^{k-1}})={\phi_q(x^{p^k})\over \phi_q(x^{p^{k-1}})}.\]

Rewriting the right-hand side of equation \eqref{eq:piq} yields

\[\phi_q(x^{p^{k-1}})\phi_{p^kq}(x)=\phi_q(x^{p^{k-1}}){\phi_q(x^{p^k})\over \phi_q(x^{p^{k-1}})}=\phi_q(x^{p^k})\]
as desired.
\end{proof}

\begin{lemma}\label{thm:josh_pqr}
For three distinct prime numbers, $p,q,r$,
$$\phi_p(x)\cdot\phi_{pq}(x)\cdot\phi_{pr}(x)\cdot\phi_{pqr}(x)=\phi_p(x^{qr}).$$
\end{lemma}

\begin{proof}
Recall that
\begin{align*}
\phi_{pqr}(x)&=\frac{(x^p-1)(x^q-1)(x^r-1)(x^{pqr}-1)}{(x-1)(x^{pq}-1)(x^{pr}-1)(x^{qr}-1)}\\
&= \frac{x^p-1}{x-1}\cdot\frac{x^q-1}{x^{pq}-1}\cdot\frac{x^r-1}{x^{pr}-1}\cdot\frac{x^{pqr}-1}{x^{qr}-1}\\
&= \phi_p(x)\cdot\frac{1}{\phi_p(x^q)}\cdot\frac{1}{\phi_p(x^r)}\cdot\phi_p(x^{qr}).
\end{align*}
Multiplying by $\phi_p(x)\cdot\phi_{pq}(x)\cdot\phi_{pr}(x)$  we have
$$\phi_p(x)\cdot\phi_{pq}(x)\cdot\phi_{pr}(x)\cdot\phi_p(x)\cdot\frac{1}{\phi_p(x^q)}\cdot\frac{1}{\phi_p(x^r)}\cdot\phi_p(x^{qr}).$$
Using identity \eqref{eq:cyc_gen} on the first four terms yields
$$\phi_p(x^q)\cdot\phi_p(x^r)\cdot\frac{1}{\phi_p(x^q)}\cdot\frac{1}{\phi_p(x^r)}\cdot\phi_p(x^{qr}).$$
Now we can cancel out some terms and we are left with $\phi_p(x^{qr})$ as desired.

\end{proof}

Before proceeding to our main arguments and results, we discuss one final technique that we will take extensive advantage of. As we have seen, cyclotomic polynomials can be expressed as rational functions. We will often consider product of these cyclotomic polynomials, and because of the similarities in the different formulas for the cases we will be considering, a lot of cancellation occurs between common factors. For example, for primes $p$ and $q$, we have 
\[\phi_p\phi_{pq}={x^p-1\over x-1}{(x-1)(x^{pq}-1)\over (x^p-1)(x^q-1)}={x^{pq}-1\over x^q-1}.\]
We can further ``simplify" this by writing it as a product of (possibly finite) formal series:
\[{x^{pq}-1\over x^q-1}={1-x^{pq}\over 1-x^q}=(1-x^{pq})\sum_{i=0}^\infty x^{iq}.\]
Writing our functions in this way will be extremely useful in demonstrating when certain options for solutions yield negative coefficients, and thus are not solutions. 


\section{The $p^2q$ Case} \label{sec:p2q}

In this section, we address question \ref{qu:gen} in the case where $n=2$ and $m=p^2q$, where $p$ and $q$ are distinct prime numbers. 
In this case, the frequency polynomial would be
$$\left(\frac{x(x^{p^2q}-1)}{x-1}\right)^2=x^2 \phi_p^2\phi_{p^2}^2 \phi_q^2 \phi_{pq}^2  \phi_{p^2q}^2$$
Recall that if a polynomial $P(x)$ corresponds to a solution, than in this case we must have $P(1)=p^2q$. Thus, such a polynomial has the form
\[x\phi_q  \phi_p^{c_p}\phi_{p^2}^{c_{p^2}} \phi_{pq}^{c_{pq}}\phi_{p^2q}^{c_{p^2q}}\]
with integers $0\leq c_p,c_{p^2},c_{pq},c_{p^2q}\leq 2$ where $c_p+c_{p^2}=2$ and $c_i\leq 2$ for all $i$. This is because since $\phi_q(1)=q$, both $\phi_p(1)$ and $\phi_{p^2}(1)$ equal $p$, and the other cyclotomic polynomials evaluate to $1$ when $x=1$.

Thus, the Table \ref{tab:p2q_init_vectors}  represents all potential pairs of solutions, organized by the sum of the exponents. We omit the standard solution, where all exponents equal 1.
\begin{center}
\begin{table}[h]
\begin{tabular}{ |c c c c c c c c c| } 
 \hline
 $c_{p}$ & $c_{p^2}$ & $c_{pq}$ & $c_{p^2q}$ & $\leftrightarrow$ & $c_{p}$ & $c_{p^2}$ & $c_{pq}$ & $c_{p^2q}$  \\ \hline
 1 & 1 & 0 & 0 & $\leftrightarrow$ & 1 & 1 & 2 & 2\\ 
 1 & 1 & 0 & 1 & $\leftrightarrow$ & 1 & 1 & 2 & 1\\ 
 1 & 1 & 1 & 0 & $\leftrightarrow$ & 1 & 1 & 1 & 2\\ 
 1 & 1 & 0 & 2 & $\leftrightarrow$ & 1 & 1 & 2 & 0\\ 
 2 & 0 & 0 & 0 & $\leftrightarrow$ & 0 & 2 & 2 & 2\\ 
 2 & 0 & 0 & 1 & $\leftrightarrow$ & 0 & 2 & 2 & 1\\ 
 2 & 0 & 1 & 0 & $\leftrightarrow$ & 0 & 2 & 1 & 2\\ 
 2 & 0 & 1 & 1 & $\leftrightarrow$ & 0 & 2 & 1 & 1\\ 
 2 & 0 & 0 & 2 & $\leftrightarrow$ & 0 & 2 & 2 & 0\\ 
 2 & 0 & 2 & 0 & $\leftrightarrow$ & 0 & 2 & 0 & 2\\ 
 2 & 0 & 1 & 2 & $\leftrightarrow$ & 0 & 2 & 1 & 0\\ 
 2 & 0 & 2 & 1 & $\leftrightarrow$ & 0 & 2 & 0 & 1\\ 
 2 & 0 & 2 & 2 & $\leftrightarrow$ & 0 & 2 & 0 & 0\\ 
 \hline
\end{tabular}
\caption{The possible exponents for the solutions.}\label{tab:p2q_init_vectors}
\end{table}
\end{center}
We shall refer to the tuple $(c_{p},c_{p^2},c_{pq},c_{p^2q})$ as an \textit{exponent vector} for the corresponding polynomial. Both $(c_{p},c_{p^2},c_{pq},c_{p^2q})$ and $(2-c_{p},2-c_{p^2},2-c_{pq},2-c_{p^2q})$ must yield polynomials with positive coefficients for either (and therefore both) to be considered a solution.

The following Lemma, and consequential Corollary, serves as a useful tool to determine one way which an exponent vector can correspond to a polynomial with negative coefficients.

\begin{lemma}\label{lem:shriram}
Let $F(x)$ be a function of the form 
\[F(x)=\prod_{i=1}^k (1-x^{n_i})^{\epsilon_i}\]
where the $n_i$ are distinct positive integers and $\epsilon_i\in \mathbb{Z}$. If there exists a $j$ for which $n_j=1$ and $\epsilon_j>0$, then $[x]F<0$.
\end{lemma}

\begin{proof}
Given such a $j$, we have $[x](1-x^{n_j})^{\epsilon_j}=-\epsilon_j$. Since $n_i\neq 1$ for all $i\neq j$, we have $[x](1-x^{n_i})^{\epsilon_i}=0$ for $i\neq j$, whether written as is if $\epsilon_i>0$ or as a series when $\epsilon_i<0$. Since each of these terms additionally have the constant $1$, this implies the desired result. 
\end{proof}

\begin{corollary}\label{cor:p2q_oneminusx}
The polynomial with exponent vector $(c_{p},c_{p^2},c_{pq},c_{p^2q})$ has a factor of $1-x$ when expressed as a (reduced) rational function if and only if 
\[c_{pq}-c_{p}-1>0.\]
\end{corollary}
\begin{proof}
 When written as rational functions as in Section \ref{sec:cyc}, $\phi_p$ and $\phi_q$ have a factor of $1-x$ on the denominator, $\phi_{pq}$ has a factor of $1-x$ on the numerator, and $\phi_{p^2}$, $\phi_{pq}$, and $\phi_{p^2q}$ do not have a factor of $1-x$. Thus, the exponent of $1-x$ in the polynomial corresponding to the given exponent vector is $c_{pq}-c_p-1$.
\end{proof}

Removing any rows satisfying Corollary \ref{cor:p2q_oneminusx} yields the Table \ref{tab:p2q_sec_vectors}. Note such a vector necessarily has the form $(0,*,2,*)$, where the entries with $*$ can be any value.

\begin{center}
\begin{table}[h]
\begin{tabular}{ |c c c c c c c c c| } 
 \hline
 $c_{p}$ & $c_{p^2}$ & $c_{pq}$ & $c_{p^2q}$ & $\leftrightarrow$ & $c_{p}$ & $c_{p^2}$ & $c_{pq}$ & $c_{p^2q}$  \\ \hline
 1 & 1 & 0 & 0 & $\leftrightarrow$ & 1 & 1 & 2 & 2\\ 
 1 & 1 & 0 & 1 & $\leftrightarrow$ & 1 & 1 & 2 & 1\\ 
 1 & 1 & 1 & 0 & $\leftrightarrow$ & 1 & 1 & 1 & 2\\ 
 1 & 1 & 0 & 2 & $\leftrightarrow$ & 1 & 1 & 2 & 0\\ 
 2 & 0 & 1 & 0 & $\leftrightarrow$ & 0 & 2 & 1 & 2\\ 
 2 & 0 & 1 & 1 & $\leftrightarrow$ & 0 & 2 & 1 & 1\\ 
 2 & 0 & 2 & 0 & $\leftrightarrow$ & 0 & 2 & 0 & 2\\ 
 2 & 0 & 1 & 2 & $\leftrightarrow$ & 0 & 2 & 1 & 0\\ 
 2 & 0 & 2 & 1 & $\leftrightarrow$ & 0 & 2 & 0 & 1\\ 
 2 & 0 & 2 & 2 & $\leftrightarrow$ & 0 & 2 & 0 & 0\\ 
 \hline
\end{tabular}
\caption{The rows from Table \ref{tab:p2q_init_vectors} not satisfying Corollary \ref{cor:p2q_oneminusx}.}\label{tab:p2q_sec_vectors}
\end{table}
\end{center}

All but three of the rows can be guaranteed to have positive coefficients by using a combination of the identities Lemma \ref{lem:product_piq}, $\phi_q\phi_{pq}=\phi_q{x^p}$, $\phi_p\phi_{pq}=\phi_p(x^q)$, and $\phi_{p^2}\phi_{qp^2}=\phi_{p^2}(x^2)$. These three cases are in the following table.

\begin{center}
\begin{table}[h]
\begin{tabular}{ |c c c c c c c c c| } 
 \hline
 $c_{p}$ & $c_{p^2}$ & $c_{pq}$ & $c_{p^2q}$ & $\leftrightarrow$ & $c_{p}$ & $c_{p^2}$ & $c_{pq}$ & $c_{p^2q}$  \\ \hline
 1 & 1 & 0 & 2 & $\leftrightarrow$ & 1 & 1 & 2 & 0\\ 
 2 & 0 & 1 & 2 & $\leftrightarrow$ & 0 & 2 & 1 & 0\\ 
 2 & 0 & 2 & 2 & $\leftrightarrow$ & 0 & 2 & 0 & 0\\ 
 \hline
\end{tabular}
\caption{The rows from Table \ref{tab:p2q_sec_vectors} which we can not guarantee have positive coefficients.}\label{tab:p2q_third_vectors}
\end{table}
\end{center}


We demonstrate explicitly that in each case, one of the pairs has negative coefficients by expressing the corresponding polynomials as a product of series. In the following proofs, we omit the formulation of these products, though the rough computation can be found in Appendix \ref{sec:comp_p2q}.

\begin{lemma}
The polynomial with exponent vector $(1,1,0,2)$, which we denote $A_{1102}$, has a negative coefficient.
\end{lemma}
\begin{proof}
After, simplification, we have
\[A_{1102}=(1-x^q)(1-x^p)^2(1-x^{p^2q})^2\left(\displaystyle \sum_{i=0}^\infty x^{i}\right)^2\left(\displaystyle \sum_{i=0}^\infty x^{ip^2}\right)\left(\displaystyle \sum_{i=0}^\infty x^{ipq}\right)^2\]
Let us discuss all the ways in which $x^{p+q-1}$ appears. 








To start, we have the following terms which come from the terms $(1-x^q)$, $(1-x^p)^2$, and $\left(\displaystyle \sum_{i=0}^\infty  x^{i}\right)^2$:
\begin{itemize}
\item $-x^q\cdot px^{p-1}=-px^{p+q-1}$
\item $-2x^p\cdot qx^{q-1}=-2qx^{p+q-1}$
\item $(p+q)x^{p+q-1}$
\end{itemize}

When $q>p$, we must account for the $x^{2p}$ term appearing in $(1-x^p)^2$. In this case, we additionally have $x^{2p}\cdot (q-p)x^{q-p-1}=(q-p)x^{p+q-1}$.
Observe the sum of these terms is precisely 
\[\begin{cases}
-q & p>q\\
-p &q>p.
\end{cases}\]

Next, note that $p+q-1<pq$ for all prime numbers $p$ and $q$. Thus, $\left(\sum x^{ipq}\right)^2$ never contributes to such coefficients.
On the other hand, it is quite possible to have $p+q-1>p^2$. However, this can only occur if $q>p$, as if $q<p$, we have that \[p+q-1<2p-1<p^2.\]
 Thus, we are done in the case where $q<p$. We proceed assuming $q>p$(and thus, the current accumulated coefficient of $x^{p+q-1}$ is $-p$)  and $p+q-1>p^2$ . The remaining coefficients for $x^{p+q-1}$ is given by the following.

\begin{align*}
[x^{p+q-1}]&\left(\sum_{i=1}^{\left\lfloor {p+q-1\over p^2}\right\rfloor} x^{ip^2} (p+q-ip^2)x^{p+q-ip^2-1}+
\sum_{i=1}^{\left\lfloor {q-1\over p^2}\right\rfloor} -2x^p x^{ip^2} (q-ip^2)x^{q-ip^2-1}\right. \\
&+\left.
\sum_{i=1}^{\left\lfloor {q-p-1\over p^2}\right\rfloor} x^{2p} x^{ip^2} (q-p-ip^2)x^{q-p-ip^2-1}\right)\\
&=\sum_{i=1}^{\left\lfloor {p+q-1\over p^2}\right\rfloor}  (p+q-ip^2)-2
\sum_{i=1}^{\left\lfloor {q-1\over p^2}\right\rfloor}  (q-ip^2)+
\sum_{i=1}^{\left\lfloor {q-p-1\over p^2}\right\rfloor}  (q-p-ip^2).
\end{align*}

Suppose that $p^2> q-1$, and thus $p^2> q-p-1$. This implies that \( {p+q-1\over p^2}<2\) so the above sum is $p+q-p^2$, which when combined with the $-p$ above is $q-p^2\leq 1$. If $p=2$, note that this implies $q=3$, and in this case $q-p^2=-1$. If $p>2$, then $q-p^2$ is an even number, and since $q\neq p^2$ this implies $q-p^2<0$. 

If $p^2\leq q-1$ but $p^2+p> q-1$, note that $1\leq \left\lfloor{q-1\over p^2}\right\rfloor\leq  \left\lfloor{p^2+p\over p^2}\right\rfloor<2$. Meanwhile, $1\leq \left\lfloor{p+q-1\over p^2}\right\rfloor\leq  \left\lfloor{p^2+2p\over p^2}\right\rfloor\leq 2$. However, the right most inequality is only an equality when $p=2$, and since $q-1<p^2+p=6$, this would imply that either $q=3$ or $q=5$, and in both cases $\left\lfloor{p+q-1\over p^2}\right\rfloor\leq 1$.
 The above sum, when combined with $-p$, gives $-q-3p^2<0$. 

If $p^2+p\leq q-1$, we have a few cases depending on if 
\[\left\lfloor{p+q-1\over p^2}\right\rfloor=\left\lfloor{q-1\over p^2}\right\rfloor\]
or
\[ \left\lfloor{q-1\over p^2}\right\rfloor=\left\lfloor{q-p-1\over p^2}\right\rfloor.\] If both are true, the above three sums becomes $0$. If only the first is true, when combined with $-p$, the sums give
\[ -q+\left\lfloor {q-1\over p^2}\right\rfloor p^2<0.\]

In the last case, when only the second is true, the three terms combined with $-p$ give
\[ q-\left\lfloor {p+q-1\over p^2}\right\rfloor p^2= q-\left(\left\lfloor {q-1\over p^2}\right\rfloor+1\right) p^2=q-\left(\left\lfloor {q\over p^2}\right\rfloor+1\right) p^2,\]
where the last equality is true since $q$ is prime and so $p^2$ does not divide $q$. Observe that $\left\lfloor {q\over p^2}\right\rfloor p^2$ is the smallest multiple of $p^2$ smaller than $q$, so the above must be negative.

\end{proof}

\begin{lemma}
The polynomial with exponent vector $(2,0,2,2)$, which we denote $A_{2022}$, has a negative coefficient.
\end{lemma}

\begin{proof}
After simplification, we have 
\[  A_{2022}=(1-x^p)^2(1-x^{p^2q})^2\left(\displaystyle \sum_{i=0}^\infty  x^{i}\right)\left(\displaystyle \sum_{i=0}^\infty  x^{iq}\right)\left(\displaystyle \sum_{i=0}^\infty  x^{ip^2}\right)^2.\]
First, if $q>p$, observe that the $x^p$ coefficient is precisely 
\[-2x^p+x^p=-x^p.\]

Otherwise, if $q<p$, observe that it is possible that the sum $\sum x^{iq}$ could contribute additional terms which make the coefficient of $x^p$ non-negative.

Let $m=p+\left\lfloor{p\over q}\right\rfloor q$. We claim the coefficient of $x^m$ is $-1$. First, note that $m<2p<p^2$ since $p>2$. Now see that the coefficient of $x^m$ is
\begin{align*}
[x^{m}]&\left(\sum_{i=0}^{\left\lfloor {m\over q}\right\rfloor}x^{iq}x^{m-iq}-2\sum_{i=0}^{\left\lfloor {m-p\over q}\right\rfloor}x^px^{iq}x^{m-p-iq} \right)\\
&= -1+\left\lfloor {p\over q}\right\rfloor+\left\lfloor {p\over q}\right\rfloor-2{\left\lfloor {p\over q}\right\rfloor}\\
&=-1.
\end{align*}

\end{proof}

\begin{lemma}
The polynomial with exponent vector $(2,0,1,2)$, which we denote $A_{2012}$, has a negative coefficient. 
\end{lemma}

\begin{proof}
After simplification, we have 
\[ A_{2012}=(1-x^p)^3(1-x^{p^2q})^2\left(\displaystyle \sum_{i=0}^\infty  x^{i}\right)^2\left(\displaystyle \sum_{i=0}^\infty  x^{ip^2}\right)^2\left(\displaystyle \sum_{i=0}^\infty  x^{ipq}\right).\]
Note that $2p-1<p^2$ and $2p-1\leq pq$ for all prime numbers $p$ and $q$. Thus, the only terms which contribute to the coefficient of $x^{2p-1}$ is $(1-x^p)^3$ and $\left(\sum x^{i}\right)^2$. Observe, between these two, the term of the form $x^{2p-1}$  is
\[-3x^p\cdot px^{p-1}+2px^{2p-1}=-px^{2p-1}.\]
whose coefficient is negative.
\end{proof}

Thus, will all these Lemmas, the following table indicates all possible solutions in the case of $p^2q$.

\begin{center}
\begin{table}[h]
\begin{tabular}{ |c c c c c c c c c| } 
 \hline
 $c_{p}$ & $c_{p^2}$ & $c_{pq}$ & $c_{p^2q}$ & $\leftrightarrow$ & $c_{p}$ & $c_{p^2}$ & $c_{pq}$ & $c_{p^2q}$  \\ \hline
 1 & 1 & 0 & 0 & $\leftrightarrow$ & 1 & 1 & 2 & 2\\ 
 1 & 1 & 0 & 1 & $\leftrightarrow$ & 1 & 1 & 2 & 1\\ 
 1 & 1 & 1 & 0 & $\leftrightarrow$ & 1 & 1 & 1 & 2\\ 
 2 & 0 & 1 & 0 & $\leftrightarrow$ & 0 & 2 & 1 & 2\\ 
 2 & 0 & 1 & 1 & $\leftrightarrow$ & 0 & 2 & 1 & 1\\ 
 2 & 0 & 2 & 0 & $\leftrightarrow$ & 0 & 2 & 0 & 2\\ 
 2 & 0 & 2 & 1 & $\leftrightarrow$ & 0 & 2 & 0 & 1\\ 
 1 & 1 & 1 & 1 & $\leftrightarrow$ & 1 & 1 & 1 & 1 \\
 \hline
\end{tabular}
\caption{The solutions for the $p^2q$ case.}\label{tab:p2q_final_vectors}
\end{table}
\end{center}

\begin{theorem}\label{thm:p2q}
Table \ref{tab:p2q_final_vectors} makes up all possible solutions for the $p^2q$. In particular, there are 15 solutions (specifically, $8$ possible pairs of dice).
\end{theorem}

\begin{example}
Since $12=2^2\cdot 3$, the following make up the complete list of labels for two $12$-sided whose sums have the same frequencies of two standard $12$-sided dice. These are ordered in the same way as in Table \ref{tab:p2q_final_vectors}. 

\begin{itemize}
\item $1,2,2,3,3,3,4,4,4,5,5,6$ and $1,4,5,7,8,9,10,11,12,14,15,18$.
\item $1,2,2,3,3,4,7,8,8,9,9,10$ and $1,3,4,5,6,7,8,9,10,11,12,14$.
\item $1 ,2 ,3 ,3 ,4 ,4 ,5 ,5 ,6 ,6 ,7 ,8 $ and $ 1 ,2 ,5 ,6 ,7 ,8 ,9 ,10 ,11 ,12 ,15 ,16$.
\item $1 ,2 ,2 ,3 ,3 ,4 ,4 ,5 ,5 ,6 ,6 ,7 $ and $1 ,3 ,5 ,7 ,7 ,9 ,9 ,11 ,11 ,13 ,15 ,17 $.
\item $ 1 ,2 ,2 ,3 ,5 ,6 ,6 ,7 ,9 ,10 ,10 ,11 $ and $1 ,3 ,3 ,5 ,5 ,7 ,7 ,9 ,9 ,11 ,11 ,13 $.
\item $ 1 ,2 ,3 ,4 ,4 ,5 ,5 ,6 ,6 ,7 ,8 ,9 $ and $ 1 ,2 ,3 ,7 ,7 ,8 ,8 ,9 ,9 ,13 ,14 ,15$.
\item $ 1 ,2 ,4 ,5 ,5 ,6 ,8 ,9 ,9 ,10 ,12 ,13$ and $ 1 ,2 ,3 ,3 ,4 ,5 ,7 ,8 ,9 ,9 ,10 ,11$.
\end{itemize}
\end{example}

\section{The $pqr$ Case} \label{sec:pqr}

In this section, we address question \ref{qu:gen} in the case where $n=2$ and $m=pqr$, where $p$, $q$, and $r$ are distinct prime numbers. 
In this case, the frequency polynomial would be
$$\left(\frac{x(x^{pqr}-1)}{x-1}\right)^2=x^2 \phi_p ^2 \phi_q ^2 \phi_r ^2 \phi_{pq} ^2 \phi_{pr} ^2 \phi_{qr} ^2 \phi_{pqr} ^2.$$

Recall that if a polynomial $P(x)$ corresponds to a solution, than in this case we must have $P(1)=pqr$. Thus, such a polynomial has the form
\[x \phi_p\phi_q\phi_r\phi_{pq}^{c_{pq}} \phi_{pr}^{c_{pr}}\phi_{qr}^{c_{qr}} \phi_{pqr}^{c_{pqr}}\]
with integers $0\leq c_p,c_{q},c_r,c_{pq},c_{pr},c_{qr},c_{pqr}\leq 2$. This is because $\phi_p(1)=p$, $\phi_q(1)=q$, $\phi_r(1)=r$, and all other cyclotomic polynomials involved equal $1$ when evaluated at $x=1$.

Thus, Table \ref{tab:init_vectors}  represents all potential pairs of solutions, organized by the sum of the exponents. We omit the standard solution, where all exponents equal 1. We additionally omit repetitions due to symmetry. For instance, 
\[\phi_p \phi_q \phi_r \phi_{pq}^{2}\phi_{qr}\phi_{pqr}\]
and
\[\phi_p \phi_q \phi_r \phi_{pq}\phi_{qr}^{2}\phi_{pqr}\]
are similar in that the later is achieved by swapping $p$ and $q$. Knowing one of these is a solution for all $p$, $q$, and $r$ is equivalent to showing the same for the other by applying an appropriate permutation of $p$, $q$, and $r$. 
 
\begin{center}
\begin{table}[h]
\begin{tabular}{ |c c c c c c c c c| } 
 \hline
 $c_{pq}$ & $c_{pr}$ & $c_{qr}$ & $c_{pqr}$ & $\leftrightarrow$ & $c_{pq}$ & $c_{pr}$ & $c_{qr}$ & $c_{pqr}$ \\ \hline
 0 & 0 & 0 & 0 & $\leftrightarrow$ & 2 & 2 & 2 & 2\\ 
 1 & 0 & 0 & 0 & $\leftrightarrow$ & 1 & 2 & 2 & 2\\ 
 0 & 0 & 0 & 1 & $\leftrightarrow$ & 2 & 2 & 2 & 1\\ 
 2 & 0 & 0 & 0 & $\leftrightarrow$ & 0 & 2 & 2 & 2\\ 
 0 & 0 & 0 & 2 & $\leftrightarrow$ & 2 & 2 & 2 & 0\\ 
 1 & 1 & 0 & 0 & $\leftrightarrow$ & 1 & 1 & 2 & 2\\
 1 & 0 & 0 & 1 & $\leftrightarrow$ & 1 & 2 & 2 & 1\\
 2 & 1 & 0 & 0 & $\leftrightarrow$ & 0 & 1 & 2 & 2\\
 2 & 0 & 0 & 1 & $\leftrightarrow$ & 0 & 2 & 2 & 1\\
 
 1 & 0 & 0 & 2 & $\leftrightarrow$ & 1 & 2 & 2 & 0\\
 1& 1 & 1 & 0 & $\leftrightarrow$ & 1 & 1 & 1 & 2\\
 1& 1 & 0 & 1 & $\leftrightarrow$ & 1 & 1 & 2 & 1\\
 
 2 & 0 & 0 & 2 & $\leftrightarrow$ & 0 & 2 & 2 & 0\\
 
 2 & 1 & 1 & 0 & $\leftrightarrow$ & 0 & 1 & 1 & 2\\
 2 & 1 & 0 & 1 & $\leftrightarrow$ & 0 & 1 & 2 & 1\\

 \hline
\end{tabular}
\caption{The possible exponents for the solutions.}\label{tab:init_vectors}
\end{table}
\end{center}

We shall refer to the tuple $(c_{pq},c_{pr},c_{qr},c_{pqr})$ as an exponent vector. Both $(c_{pq},c_{pr},c_{qr},c_{pqr})$ and $(2-c_{pq},2-c_{pr},2-c_{qr},2-c_{pqr})$ must yield polynomials with positive coefficients for either (and thus both) to be a solution. 

Recall by Lemma \ref{lem:shriram} that a polynomial has a negative coefficient if a factor of $1-x$ appears on the numerator when reduced and written as a rational function. The following is a Corollary that will allow us remove many entries appearing in Table \ref{tab:init_vectors} which have a negative coefficient.
\begin{corollary}\label{cor:oneminusx}
The polynomial with exponent vector $(c_{pq},c_{pr},c_{qr},c_{pqr})$ has a factor of $1-x$ in its rational function if and only if 
\[c_{pq}+c_{pr}+c_{qr}-c_{pqr}-3>0.\]
\end{corollary}
\begin{proof}
When written as rational functions, $\phi_p$, $\phi_q$, $\phi_r$, and $\phi_{pqr}$ have a factor of $1-x$ on the denominator, while $\phi_{pq}$, $\phi_{pr}$, $\phi_{qr}$ have a factor of $1-x$ on the numerator. Thus, the exponent of $1-x$ in the polynomial corresponding to the given exponent vector is $c_{pq}+c_{pr}+c_{qr}-c_{pqr}-3$.

\end{proof}

Due to this, after eliminating the rows satisfying the conditions in the prior Corollary, we have the following.

\begin{center}
\begin{table}[h]
\begin{tabular}{ |c c c c c c c c c| } 
 \hline
 $c_{pq}$ & $c_{pr}$ & $c_{qr}$ & $c_{pqr}$ & $\leftrightarrow$ & $c_{pq}$ & $c_{pr}$ & $c_{qr}$ & $c_{pqr}$ \\ \hline
 1 & 0 & 0 & 0 & $\leftrightarrow$ & 1 & 2 & 2 & 2\\ 
 2 & 0 & 0 & 0 & $\leftrightarrow$ & 0 & 2 & 2 & 2\\ 
 1 & 1 & 0 & 0 & $\leftrightarrow$ & 1 & 1 & 2 & 2\\
 2 & 1 & 0 & 0 & $\leftrightarrow$ & 0 & 1 & 2 & 2\\
 2 & 0 & 0 & 1 & $\leftrightarrow$ & 0 & 2 & 2 & 1\\
 
 1& 1 & 1 & 0 & $\leftrightarrow$ & 1 & 1 & 1 & 2\\
 1& 1 & 0 & 1 & $\leftrightarrow$ & 1 & 1 & 2 & 1\\
 
 
 
 2 & 1 & 0 & 1 & $\leftrightarrow$ & 0 & 1 & 2 & 1\\
 
 \hline
\end{tabular}
\caption{Table \ref{tab:init_vectors} after removing rows satisfying Corollary \ref{cor:oneminusx}.}\label{tab:sec_vectors}
\end{table}
\end{center}

We can further reduce this list, using Theorem \ref{thm:josh_pqr}, which has the following implication. 

\begin{corollary}\label{cor:josh_reduction}
If a given exponent vector, after successive differences with at least one of $(1,1,0,1)$, $(1,0,1,1)$, or $(0,1,1,1)$, has at most one $1$, the corresponding polynomial has positive coefficients. 
\end{corollary}

\begin{proof}
For each removal of such a vector which is legal, this corresponds to an application of Theorem \ref{thm:josh_pqr}. For instance, $(c_{pq},c_{pr},c_{qr},c_{pqr})-(1,1,0,1)$ corresponds to applying the identity $\phi_p\phi_{pq}\phi_{pr}\phi_{pqr}=\phi_p(x^{qr})$, which has positive coefficients. If after applying all differences we end with vectors such as $(1,0,0,0)$, $(0,1,0,0)$, or $(0,0,1,0)$, these can be combined with $\phi_p$, $\phi_q$, or $\phi_r$ using the identity from equation \ref{eq:cyc_gen} to achieve a polynomial with positive coefficients. 
\end{proof}

After applying Corollary \ref{cor:josh_reduction} or equation \ref{eq:cyc_gen} to entries in Table \ref{tab:sec_vectors}, we are left with the following which can not be guaranteed to have positive coefficients.

\begin{center}
\begin{table}[h]
\begin{tabular}{ |c c c c c c c c c| } 
 \hline
 $c_{pq}$ & $c_{pr}$ & $c_{qr}$ & $c_{pqr}$ & $\leftrightarrow$ & $c_{pq}$ & $c_{pr}$ & $c_{qr}$ & $c_{pqr}$ \\ \hline
 2 & 0 & 0 & 0 & $\leftrightarrow$ & 0 & 2 & 2 & 2\\ 
 2 & 1 & 0 & 0 & $\leftrightarrow$ & 0 & 1 & 2 & 2\\
 2 & 0 & 0 & 1 & $\leftrightarrow$ & 0 & 2 & 2 & 1\\
 
 1& 1 & 1 & 0 & $\leftrightarrow$ & 1 & 1 & 1 & 2\\
 
 
 
 
 \hline
\end{tabular}
\caption{Table \ref{tab:sec_vectors} after removing rows where both solutions can be guaranteed to have positive coefficients using Corollary \ref{cor:josh_reduction} and the identity in equation \ref{eq:cyc_gen}.}\label{tab:third_vectors}
\end{table}
\end{center}

In what remains in Table \ref{tab:third_vectors}, we claim each row has an entry corresponding to a polynomial with a negative coefficient. We prove this in four separate results by expressing the corresponding polynomials as products of series. In the following proofs, we omit the formulation of these products, though the rough computation can be found in Appendix \ref{sec:comp_pqr}.

\begin{lemma}
The polynomial corresponding to exponent vector $(0,2,2,2)$, which we denote $A_{0222}$, has a negative coefficient.
\end{lemma}
\begin{proof}

After simplification, one can write $A_{0222}$ as the following product: 
\[(1-x^p)(1-x^q)(1-x^{pqr})^2\left(\displaystyle \sum_{i=0}^\infty  x^i\right)\left(\displaystyle \sum_{i=0}^\infty  x^{ir}\right)\left(\displaystyle \sum_{i=0}^\infty  x^{ipq}\right)\]
We proceed by cases based on the relative values of $p$, $q$, and $r$. 

Suppose first $r$ is the largest. In this case, we consider $[x^M]A_{0222}$ where $M=\max(p,q)$. Note that $M<r$, $M<pq$, and $M<pqr$, so we have
\begin{align*}
[x^M]A_{0222}&=[x^M](1-x^p)(1-x^q)\left(\displaystyle \sum_{i=0}^\infty  x^i\right)\\
&=[x^M](-x^px^{M-p}-x^qx^{M-q}+x^M)\\
&=-1.
\end{align*}

 Suppose now that $p$ or $q$ is the largest. Note that $A_{0222}$ is symmetric in $p$ and $q$, so it suffices to prove this in just one case. To this end, we will assume $p$ is the largest. We proceed by cases on the relative values of $r$ and $q$.
\begin{itemize}
\item[\textbf{Case 1:}] 
First, suppose $r<q<p$. We claim there exists a $j\in[1,r-1]$ so that $[x^{p+q-j}]A_{0222}$ is negative. Note that $p+q-j<pqr$ and $p+q-j<pq$, so  

\begin{align*}
[x^{p+q-j}]A_{0222}&=[x^{p+q-j}](1-x^p)(1-x^q)\left(\displaystyle \sum_{i=0}^\infty  x^i\right)\left(\displaystyle \sum_{i=0}^\infty  x^{ir}\right)\\
&=[x^{p+q-j}]\left(-\sum_{i=0}^{\left\lfloor{q-j\over r} \right\rfloor} x^px^{ir}x^{q-j-ir}-\sum_{i=0}^{\left\lfloor{p-j\over r} \right\rfloor} x^qx^{ir}x^{p-j-ir}\right.\\
&\left.+\sum_{i=0}^{\left\lfloor{p+q-j\over r} \right\rfloor} x^{ir}x^{p+q-j-ir}\right)\\
&=-\left\lfloor{q-j\over r} \right\rfloor-\left\lfloor{p-j\over r} \right\rfloor+\left\lfloor{p+q-j\over r} \right\rfloor-1.
\end{align*}
Note that $r$, $q$, and $p$ are all prime, so $r$ does not divide the other two. To this end, let $j$ be selected minimally so that $q-j$ or $p-j$ is divisible by $r$. By symmetry, we may suppose that $j$ is minimized by making $q-j$ divisible by $r$. We necessarily have that $j\in[1,r-1]$. Then using the identity that 
\[\left\lfloor x+n\right\rfloor=\left\lfloor x\right\rfloor +n\]
for all integers $n$, we have know that 
\[\left\lfloor{p+q-j\over r} \right\rfloor=\left\lfloor{p\over r} \right\rfloor+{q-j\over r}.\]
Due to our choice of $j$, this means that \[\left\lfloor{p\over r} \right\rfloor=\left\lfloor{p-j\over r} \right\rfloor\]

From this, we can conclude that 
\[-\left\lfloor{q-j\over r} \right\rfloor-\left\lfloor{p-j\over r} \right\rfloor+\left\lfloor{p+q-j\over r} \right\rfloor=0\]
and so $[x^{p+q-j}]A_{0222}=-1$.

\item[\textbf{Case 2:}] 
Now suppose that $q<r<p$. In this case, we choose a non-negative integer $j<q$ and study $[x^{p+j}]A_{0222}$. Observe that 
\begin{align*}
[x^{p+j}]A_{0222}&=[x^{p+j}](1-x^p)(1-x^q)\left(\displaystyle \sum_{i=0}^\infty  x^i\right)\left(\displaystyle \sum_{i=0}^\infty  x^{ir}\right)\\
&=[x^{p+j}]\left(-x^px^j-\sum_{i=0}^{\left\lfloor {p+j-q\over r}\right\rfloor}x^qx^{ir}x^{p+j-q-ir}+\sum_{i=0}^{\left\lfloor {p+j\over r}\right\rfloor}x^{ir}x^{p+j-ir}\right)\\
&=-\left\lfloor {p+j-q\over r}\right\rfloor+\left\lfloor {p+j\over r}\right\rfloor-1.
\end{align*}

If the interval of integers $[p-q+1,p]$ does not have a multiple of $r$, we have 
\[\left\lfloor {p-q\over r}\right\rfloor=\left\lfloor {p\over r}\right\rfloor,\]
and so in this case, by letting $j=0$, we have $[x^p]A_{0222}=-1$. Otherwise, choose $j$ so that $j\leq q$ and $r$ divides $(p+j-q)$. Note that, in fact, $j\leq q-1$, since if we let $j=q$ this would imply that $r$ divides $p$. Since $q<r$, this implies 
\[\left\lfloor {p-q+j\over r}\right\rfloor=\left\lfloor {p+j\over r}\right\rfloor,\]
so $[x^{p+j}]A_{0222}=-1.$

\end{itemize}

\end{proof}

\begin{lemma}\label{lem:pqr_0122}
The polynomial corresponding to exponent vector $(0,1,2,2)$, which we denote $A_{0122}$, has negative coefficients. 
\end{lemma}

\begin{proof}
After simplification, $A_{0122}$ is
\[(1-2x^p+x^{2p})(1-x^q)(1-x^{pqr})^2\left(\displaystyle \sum_{i=0}^\infty  (i+1)x^i\right)\left(\displaystyle \sum_{i=0}^\infty  x^{ipr}\right)\left(\displaystyle \sum_{i=0}^\infty  x^{ipq}\right)^2.\]

We proceed by cases on the relative values of $p$, $q$, and $r$. First, suppose that $r$ is the largest. Note that $p+q-1<pq<pr$. Also $p+q-1<2p$ if and only if $q-1<p$ if and only if $q< p$ (since $q$ and $p$ are distinct). Thus, $[x^{p+q-1}]A_{0122}$ depends on the relative value of $p$ and $q$. Let $\delta(\bullet)$ be the kronecker delta, returning 0 if $\bullet$ is false and 1 if $\bullet$ is true. Then
\begin{align*}
[x^{p+q-1}]A_{0122}&=[x^{p+q-1}](1-2x^p+x^{2p})(1-x^p)\left(\displaystyle \sum_{i=0}^\infty (i+1)x^i\right)\\
&=[x^{p+q-1}](-2x^p\cdot qx^{q-1}-x^q\cdot px^{p-1}+(q+p)x^{p+q-1}+\delta(p<q)x^{2p}(q-p)x^{q-p-1})\\
&=-q+\delta(p<q)(q-p)
\end{align*}
which is less than 0 regardless of the relative values of $p$ and $q$.

Now suppose $p$ is the largest among the prime numbers. Note we still have that $p+q-1<pr$, $p+q-1<pq$, and this time we always have $p+q-1<2p$. Thus, $[x^{p+q-1}]A_{0122}<0$ in this case as well. 

All that remains is the case where $q$ is the largest. We again study the coefficient of $x^{q+p-1}$. We still have that $p+q-1<pq$. However,  $p+q-1>2p$, and it is possible that $pr<q+p-1$. We proceed carefully in two cases. 

\begin{itemize}
\item[\textbf{Case 1:}] Suppose $pr>q$.  Note we can not have $ipr<q+p-1$ for $i\geq 2$ as then $2\leq i<{q+p-1\over pr}<2$, since $pr>q$. Thus, we have
\begin{align*}
[x^{p+q-1}]A_{0122}&=[x^{p+q-1}](1-2x^p+x^{2p})(1-x^q)\big(\sum(i+1)x^i\big)\left(\displaystyle \sum_{i=0}^\infty  x^{ipr}\right)\left(\displaystyle \sum_{i=0}^\infty  x^{ipq}\right)^2\\
&=[x^{p+q-1}](-2x^p\cdot qx^{q-1}-x^q\cdot px^{p-1}+(q+p)x^{p+q-1}
\\&+x^{2p}(q-p)x^{q-p-1}+\delta(pr\leq q+p-1)x^{pr}(q+p-pr)x^{q+p-pr-1}\\
&=-p+\delta(pr\leq q+p-1)(q+p-pr)
\end{align*}
which is always negative since $q<pr$.

\item[\textbf{Case 2:}] Suppose now that $q>pr$. (Note this implies $q>2p$ as well.) In this case, our coefficient can be much more complicated. Indeed, this time we have

\begin{align*}
[x^{p+q-1}]A_{0122}&=[x^{p+q-1}]\left(-2\sum_{i=0}^{\left\lfloor {q-1\over pr}\right\rfloor}  x^px^{ipr}(q-ipr)x^{q-ipr-1}-x^q\cdot px^{p-1}\right.\\
&\left.+\sum_{i=0}^{\left\lfloor {q-p-1\over pr}\right\rfloor}  x^{2p}x^{ipr}(q-p-ipr)x^{q-p-ipr-1}+\sum_{i=0}^{\left\lfloor {q+p-1\over pr}\right\rfloor}  x^{ipr}(q+p-ipr)x^{q+p-ipr-1}\right)\\
&=-p-2\sum_{i=0}^{\left\lfloor {q-1\over pr}\right\rfloor}  (q-ipr)+\sum_{i=0}^{\left\lfloor {q-p-1\over pr}\right\rfloor}  (q-p-ipr)+\sum_{i=0}^{\left\lfloor {q+p-1\over pr}\right\rfloor}  (q+p-ipr).
\end{align*}

Note that the limits of the sums must pairwise differ by at most one, and together be at most two distinct integers, since the numerators span the interval $[q-p-1,q+p-1]$, which has length $2p$ and the denominators are $pr$ which is greater than or equal to $ 2p$. That is, we must have \[{\left\lfloor {q+p-1\over pr}\right\rfloor}={\left\lfloor {q-1\over pr}\right\rfloor}\]
or  \[{\left\lfloor {q-1\over pr}\right\rfloor}={\left\lfloor {q-p-1\over pr}\right\rfloor}.\]

If both equalities hold, the above sum becomes $-p$. If we only have ${\left\lfloor {q+p-1\over pr}\right\rfloor}={\left\lfloor {q-1\over pr}\right\rfloor}$, then the above sum becomes 
\[-(q-{\left\lfloor {q-1\over pr}\right\rfloor}pr)\]
which is negative since $pr$ does not divide $q$, so necessarily $\left\lfloor {q-1\over pr}\right\rfloor\leq \left\lfloor {q\over pr}\right\rfloor<{q\over pr}$.
Finally, if we only have ${\left\lfloor {q-1\over pr}\right\rfloor}={\left\lfloor {q-p-1\over pr}\right\rfloor}$, the above sum becomes 
\[q-\left\lfloor {q+p-1\over pr}\right\rfloor pr=q-\left(\left\lfloor {q-1\over pr}\right\rfloor +1\right)pr=q-\left\lfloor {q\over pr}\right\rfloor pr -pr.\]

The final equality is true since $pr$ does not divide $q$. This must be negative since $\left\lfloor {q\over pr}\right\rfloor pr $ is the largest integer multiple of $pr$ smaller than $q$.
\end{itemize}

\end{proof}

\begin{lemma}
The polynomial corresponding to exponent vector $(2,0,0,1)$, which we denote $A_{2001}$, has a negative coefficient. 
\end{lemma}

\begin{proof}
After simplification, the corresponding polynomial is 


\[(1-2x^r+x^{2r})(1-x^{pq})(1-x^{pqr})\left(\displaystyle \sum_{i=0}^\infty  (i+1)x^i\right)\left(\displaystyle \sum_{i=0}^\infty  x^{ipr}\right)\left(\displaystyle \sum_{i=0}^\infty  x^{iqr}\right)\]

We will study the coefficient of $x^{r+pq-1}$. First, we suppose $r$ is the largest among the three prime numbers. In this case, note that $pq<pr$ and  ${r+pq\over pr}<2$. We have the same identities with $qr$ in place of $pr$. Thus, we have 
\begin{align*}
[x^{pq+r-1}] A_{2001}&=[x^{pq+r-1}]\left(-2x^r\cdot pqx^{pq-1}-x^{pq}\cdot rx^{r-1}+(pq+r)x^{pq+r-1}\right.\\
& \left. +\left\lfloor{r+pq\over pr}\right\rfloor x^{pr}\cdot (r+pq-pr)x^{r+pq-pr-1} +\left\lfloor{r+pq\over qr}\right\rfloor x^{qr}\cdot (r+pq-qr)x^{pq-qr-1} \right)\\
&\leq [x^{pq+r-1}]\left(-2x^r\cdot pqx^{pq-1}-x^{pq}\cdot rx^{r-1}+(pq+r)x^{pq+r-1}\right.\\
& \left. + x^{pr}\cdot (r+pq-pr)x^{r+pq-pr-1} + x^{qr}\cdot (r+pq-qr)x^{pq-qr-1} \right)\\
&=pq+2r-pr-qr
\end{align*}

which is negative. 

If $r$ is not the largest prime then we have 
{\small
\begin{align*}
[x^{pq+r-1}] A_{2001}&=[x^{pq+r-1}]\left(-2x^r\cdot pqx^{pq-1}-x^{pq}\cdot rx^{r-1}+(pq+r)x^{pq+r-1}+x^{2r}\cdot (pq-r)x^{pq-r-1}\right.\\
&  -2\sum_{i=1}^{\left\lfloor {pq-1\over pr}\right\rfloor}x^rx^{ipr}\cdot (pq-ipr)x^{pq-ipr-1} -2\sum_{i=1}^{\left\lfloor {pq-1\over qr}\right\rfloor}x^rx^{iqr}\cdot (pq-iqr)x^{pq-iqr-1} \\
&  +\sum_{i=1}^{\left\lfloor {pq+r-1\over pr}\right\rfloor}x^{ipr}\cdot (r+pq-ipr)x^{r+pq-ipr-1} +\sum_{i=1}^{\left\lfloor {pq+r-1\over qr}\right\rfloor}x^{iqr}\cdot (r+pq-iqr)x^{r+pq-iqr-1} \\
&  +\sum_{i=1}^{\left\lfloor {pq-r-1\over pr}\right\rfloor}x^{2r}x^{ipr}\cdot (pq-ipr-r)x^{pq-ipr-r-1}\\
&\left.+\sum_{i=1}^{\left\lfloor {pq-r-1\over qr}\right\rfloor}x^{2r}x^{iqr}\cdot (pq-r-iqr)x^{pq-iqr-r-1} \right)\\
&=- r  -2\sum_{i=1}^{\left\lfloor {pq-1\over pr}\right\rfloor}(pq-ipr)
-2\sum_{i=1}^{\left\lfloor {pq-1\over qr}\right\rfloor} (pq-iqr) +\sum_{i=1}^{\left\lfloor {pq+r-1\over pr}\right\rfloor}(r+pq-ipr) \\
&+\sum_{i=1}^{\left\lfloor {pq+r-1\over qr}\right\rfloor} (r+pq-iqr) +\sum_{i=1}^{\left\lfloor {pq-r-1\over pr}\right\rfloor} (pq-ipr-r)+\sum_{i=1}^{\left\lfloor {pq-r-1\over qr}\right\rfloor}(pq-r-iqr).
\end{align*}
}
If one groups the sums with $pr$ together, one can see that following Case 2 of Lemma \ref{lem:pqr_0122} that these sums are negative by replacing $q$ with $pq$ and $p$ with $r$. Consequently, grouping the sums with $qr$ together also yields a negative value, and thus we've demonstrated the desired coefficient is negative.  
\end{proof}

\begin{lemma}
The polynomial corresponding to exponent vector which we denote $A_{1112}$, has negative coefficients. 
\end{lemma}

\begin{proof}
After simplification, $$A_{1112}=(1-x^p)(1-x^q)(1-x^r)(1-x^{pqr})^2\left(\displaystyle \sum_{i=0}^\infty {(i+1)x^i}\right)\left(\displaystyle \sum_{i=0}^\infty {x^{ipq}}\right)\left(\displaystyle \sum_{i=0}^\infty {x^{ipr}}\right)\left(\displaystyle \sum_{i=0}^\infty {x^{iqr}}\right).$$
Observe this polynomial is completely symmetric in the prime numbers $p$, $q$, and $r$. Thus, we may assume without loss of generality that $p<q<r$.

We claim there exists a $j< p+q$ so that $[x^{r+j}]A_{1112}<0$. First, note that \[qr>pr=(p-1)r+r> (p-1)(q+1)+r=q(p-1)+p-1+r\geq q+p+r-1.\]
Thus, the sums involving $x^{ipr}$ and $x^{iqr}$ do not contribute to the coefficient of $x^{r+j}$ in $A_{1112}$ when $j<p+q$. However, it is possible the sum involving $x^{ipq}$ is involved. We proceed by cases based on how $pq$ is related to $r$. 

\begin{itemize}
\item[\textbf{Case 1:}] Suppose $pq>r$. Pick $j=\max(p+q-r,0)$. Note this is zero if and only if $p+q\leq r$. Additionally, since $p+q<pq$, we always have $r+j<pq$. Thus,
\begin{align*}
[x^{r+j}]A_{1112}&=[x^{r+j}]\left(-x^p(r+j-p+1)x^{r+j-p}-x^q(r+j-q+1)x^{r+j-q}-x^r(j+1)x^j\right.\\
&\left.+(r+j+1)x^{r+j}+x^px^q(r+j-p-q+1)x^{r+j-p-q}\right)\\
&=-(j+1)
\end{align*}

\item[\textbf{Case 2:}] Now we suppose that $pq<r$. Note necessarily we have $p+q<r$. Observe that so long as $j<p+q$, we have

\begin{align*}
[x^{r+j}]A_{1112}&=[x^{r+j}]\left(-x^p(r+j-p+1)x^{r+j-p}-x^q(r+j-q+1)x^{r+j-q}-x^r(j+1)x^j\right.\\
&+(r+j+1)x^{r+j}+x^px^q(r+j-p-q+1)x^{r+j-p-q}\\
& +\sum_{i=0}^{\left\lfloor {r+j\over pq}\right\rfloor} x^{ipq}(r+j-ipq+1)x^{r+j-ipq}\\
&-\sum_{i=0}^{\left\lfloor {r+j-p\over pq}\right\rfloor} x^px^{ipq}(r+j-p-ipq+1)x^{r+j-p-ipq}\\
&\left.-\sum_{i=0}^{\left\lfloor {r+j-q\over pq}\right\rfloor} x^qx^{ipq}(r+j-q-ipq+1)x^{r+j-q-ipq}\right)\\
&=\left(-(j+1)+ \sum_{i=0}^{\left\lfloor {r+j\over pq}\right\rfloor} (r+j-ipq+1) -\sum_{i=0}^{\left\lfloor {r+j-p\over pq}\right\rfloor} (r+j-p-ipq+1)\right.\\
&\left.-\sum_{i=0}^{\left\lfloor {r+j-q\over pq}\right\rfloor} (r+j-q-ipq+1)\right).
\end{align*}
\end{itemize}

Observe the interval of integers $[r+j-p-q,r+j]$ has length $p+q$. Thus, there exists an integer $j\leq p+q$, chosen minimally, for which 
\[\left\lfloor{r+j-p-q\over pq}\right\rfloor =\left\lfloor{r+j-q\over pq}\right\rfloor =\left\lfloor{r+j-p\over pq}\right\rfloor =\left\lfloor{r+j\over pq}\right\rfloor\]
since $p+q<pq$. Note however that $j\neq p+q$, as this otherwise implies that $r$ is not prime, so it must be that $j<p+q$. In this case, the above reduces to 
\[-(j+1)-\sum_{i=0}^{\left\lfloor {r+j-p-q\over pq}\right\rfloor} (r+j-p-q-ipq+1).\]
The terms of the sum indexed by $i$ are always positive due to the upper bound on $i$. Thus, the above is always negative.
\end{proof}

In light of the prior four Lemmas, we may conclude that the all possible solutions correspond to one of the rows in the following table.
 
\begin{center}
\begin{table}[h]
\begin{tabular}{ |c c c c c c c c c| } 
 \hline
 $c_{pq}$ & $c_{pr}$ & $c_{qr}$ & $c_{pqr}$ & $\leftrightarrow$ & $c_{pq}$ & $c_{pr}$ & $c_{qr}$ & $c_{pqr}$ \\ \hline
 1 & 0 & 0 & 0 & $\leftrightarrow$ & 1 & 2 & 2 & 2\\ 
 1 & 1 & 0 & 0 & $\leftrightarrow$ & 1 & 1 & 2 & 2\\
 
 1& 1 & 0 & 1 & $\leftrightarrow$ & 1 & 1 & 2 & 1\\
 
 
 
 2 & 1 & 0 & 1 & $\leftrightarrow$ & 0 & 1 & 2 & 1\\
 
 1 & 1 & 1 & 1 & $\leftrightarrow$ & 1 & 1 & 1 & 1\\
 \hline
\end{tabular}
\caption{The solutions for two dice in the $pqr$ case. }\label{tab:solutions_pqr}
\end{table}
\end{center}

\begin{theorem}\label{thm:pqr}
Table \ref{tab:solutions_pqr} has the complete set of solutions (up to symmetry) in the case of two dice with sides $pqr$, where $p$, $q$, and $r$ are distinct prime numbers. In particular, there are $25$ solutions, with $13$ possible pairs of dice.
\end{theorem}
\begin{proof}
The first four rows in the corresponding table each give three unique pairs of solutions by considering all rearrangements of the values for $c_{pq}$, $c_{pr}$, $c_{qr}$.
\end{proof}

\begin{example}
The smallest example of the case in this section corresponds to $30=2\cdot 3\cdot 5$.  As these dice have many labels, instead of writing a number multiple times in a list, we will use the notation $n^{(k)}$ to notate $n,n,...,n$ (where $n$ appears $k$ times). The following is a complete list of labelings for this case.

The following come from the first row of Table \ref{tab:solutions_pqr}.
{\small
\begin{itemize}
\item $1 ,2 ^{( 2 )},3 ^{( 3 )},4 ^{( 4 )},5 ^{( 5 )},6 ^{( 5 )},7 ^{( 4 )},8 ^{( 3 )},9 ^{( 2 )},10$ and \\ $1 ,6 ,7 ,11 ,12 ,13 ,16 ,17 ,18 ,19 ,21 ,22 ,23 ,24 ,25 ,26 ,27 ,28 ,29 ,30 ,32 ,33 ,34 ,35 ,38 ,39 ,40 ,44 ,45 ,50 $
\item $1 ,2 ^{( 2 )},3 ^{( 3 )},4 ^{( 3 )},5 ^{( 3 )},6 ^{( 3 )},7 ^{( 3 )},8 ^{( 3 )},9 ^{( 3 )},10 ^{( 3 )},11 ^{( 2 )},12$ and \\
$1 ,4 ,7 ,10 ,11 ,13 ,14 ,16 ,17 ,19 ,20 ,21 ,22 ,23 ,24 ,25 ,26 ,27 ,28 ,29 ,30 ,32 ,33 ,35 ,36 ,38 ,39 ,42 ,45 ,48$
\item $1 ,2 ^{( 2 )},3 ^{( 2 )},4 ^{( 2 )},5 ^{( 2 )},6 ^{( 2 )},7 ^{( 2 )},8 ^{( 2 )},9 ^{( 2 )},10 ^{( 2 )},11 ^{( 2 )},12 ^{( 2 )},13 ^{( 2 )},14 ^{( 2 )},15 ^{( 2 )},16 $ and \\
$1 ,3 ,5 ,7 ,9 ,11 ,13 ,15 ,16 ,17 ,18 ,19 ,20 ,21 ,22 ,23 ,24 ,25 ,26 ,27 ,28 ,29 ,30 ,32 ,34 ,36 ,38 ,40 ,42 ,44 $

\end{itemize}}

The following come from the second row of of Table \ref{tab:solutions_pqr}.
{\small
\begin{itemize}

\item $1 ,2 ,3 ^{( 2 )},4 ^{( 2 )},5 ^{( 3 )},6 ^{( 3 )},7 ^{( 3 )},8 ^{( 3 )},9 ^{( 3 )},10 ^{( 3 )},11 ^{( 2 )},12 ^{( 2 )},13 ,14$ and \\ $1 ,2 ,7 ,8 ,11 ,12 ,13 ,14 ,17 ,18 ,19 ,20 ,21 ,22 ,23 ,24 ,25 ,26 ,27 ,28 ,29 ,30 ,33 ,34 ,35 ,36 ,39 ,40 ,45 ,46$

\item $1 ,2 ,3 ,4 ^{( 2 )},5 ^{( 2 )},6 ^{( 2 )},7 ^{( 2 )},8 ^{( 2 )},9 ^{( 2 )},10 ^{( 2 )},11 ^{( 2 )},12 ^{( 2 )},13 ^{( 2 )},14 ^{( 2 )},15 ^{( 2 )},16 ,17 ,18$ and \\
$1 ,2 ,3 ,7 ,8 ,9 ,13 ,14 ,15 ,16 ,17 ,18 ,19 ,20 ,21 ,22 ,23 ,24 ,25 ,26 ,27 ,28 ,29 ,30 ,34 ,35 ,36 ,40 ,41 ,42 $

\item $1 ,2 ,3 ,4 ,5 ,6 ^{( 2 )},7 ^{( 2 )},8 ^{( 2 )},9 ^{( 2 )},10 ^{( 2 )},11 ^{( 2 )},12 ^{( 2 )},13 ^{( 2 )},14 ^{( 2 )},15 ^{( 2 )},16 ,17 ,18 ,19 ,20 $ and\\ $1 ,2 ,3 ,4 ,5 ,11 ,12 ,13 ,14 ,15 ,16 ,17 ,18 ,19 ,20 ,21 ,22 ,23 ,24 ,25 ,26 ,27 ,28 ,29 ,30 ,36 ,37 ,38 ,39 ,40$
\end{itemize}}

The following come from the third row of Table \ref{tab:solutions_pqr}.
{\small
\begin{itemize}

\item $1 ,2 ^{( 2 )},3 ^{( 3 )},4 ^{( 3 )},5 ^{( 3 )},6 ^{( 2 )},7 ,16 ,17 ^{( 2 )},18 ^{( 3 )},19 ^{( 3 )},20 ^{( 3 )},21 ^{( 2 )},22$ and\\ $1 ,4 ,6 ,7 ,9 ,10 ,11 ,12 ,13 ,14 ,15 ,16 ,17 ,18 ,19 ,20 ,21 ,22 ,23 ,24 ,25 ,26 ,27 ,28 ,29 ,30 ,32 ,33 ,35 ,38 $
\item $1 ,2 ^{( 2 )},3 ^{( 2 )},4 ^{( 2 )},5 ^{( 2 )},6 ,11 ,12 ^{( 2 )},13 ^{( 2 )},14 ^{( 2 )},15 ^{( 2 )},16 ,21 ,22 ^{( 2 )},23 ^{( 2 )},24 ^{( 2 )},25 ^{( 2 )},26$ and \\ 
$
1 ,3 ,5 ,6 ,7 ,8 ,9 ,10 ,11 ,12 ,13 ,14 ,15 ,16 ,17 ,18 ,19 ,20 ,21 ,22 ,23 ,24 ,25 ,26 ,27 ,28 ,29 ,30 ,32 ,34 $

\item $1 ,2 ^{( 2 )},3 ^{( 2 )},4 ,7 ,8 ^{( 2 )},9 ^{( 2 )},10 ,13 ,14 ^{( 2 )},15 ^{( 2 )},16 ,19 ,20 ^{( 2 )},21 ^{( 2 )},22 ,25 ,26 ^{( 2 )},27 ^{( 2 )},28$ and \\
$1 ,3 ,4 ,5 ,6 ,7 ,8 ,9 ,10 ,11 ,12 ,13 ,14 ,15 ,16 ,17 ,18 ,19 ,20 ,21 ,22 ,23 ,24 ,25 ,26 ,27 ,28 ,29 ,30 ,32 $
\end{itemize}}

The following come from the fourth row of Table \ref{tab:solutions_pqr}.
{\small
\begin{itemize}
\item $1 ,2 ,3 ^{( 2 )},4 ^{( 2 )},5 ^{( 3 )},6 ^{( 2 )},7 ^{( 2 )},8 ,9 ,16 ,17 ,18 ^{( 2 )},19 ^{( 2 )},20 ^{( 3 )},21 ^{( 2 )},22 ^{( 2 )},23 ,24$ and\\ $1 ,2 ,6 ,7 ^{( 2 )},8 ,11 ,12 ^{( 2 )},13 ^{( 2 )},14 ,17 ,18 ^{( 2 )},19 ^{( 2 )},20 ,23 ,24 ^{( 2 )},25 ^{( 2 )},26 ,29 ,30 ^{( 2 )},31 ,35 ,36$
\item $1 ,2 ,4 ,5 ,7 ,8 ,10 ,11 ^{( 2 )},12 ,13 ,14 ^{( 2 )},15 ,17 ,18 ,20 ,21 ^{( 2 )},22 ,23 ,24 ^{( 2 )},25 ,27 ,28 ,30 ,31 ,33 ,34 $ and\\
$1 ,2 ,3 ^{( 2 )},4 ,5 ^{( 2 )},6 ,7 ^{( 2 )},8 ,9 ^{( 2 )},10 ,11 ,16 ,17 ,18 ^{( 2 )},19 ,20 ^{( 2 )},21 ,22 ^{( 2 )},23 ,24 ^{( 2 )},25 ,26$
\item $1 ,2 ,6 ,7 ^{( 2 )},8 ,11 ,12 ^{( 2 )},13 ^{( 2 )},14 ,17 ,18 ^{( 2 )},19 ^{( 2 )},20 ,23 ,24 ^{( 2 )},25 ^{( 2 )},26 ,29 ,30 ^{( 2 )},31 ,35 ,36 $ and\\  $1 ,2 ,3 ^{( 2 )},4 ^{( 2 )},5 ^{( 3 )},6 ^{( 2 )},7 ^{( 2 )},8 ,9 ,16 ,17 ,18 ^{( 2 )},19 ^{( 2 )},20 ^{( 3 )},21 ^{( 2 )},22 ^{( 2 )},23 ,24$\\
\end{itemize}}

\end{example}

\section{Solutions with Different Size}\label{sec:new_sizes}

In this section we explore Question \ref{qu:dif_size}, which investigates if it is possible to find solutions which do not have the same size as the original two standard $m$-sided dice. We have the following. 

\begin{theorem}
Let $m=ab$, where $a$ and $b$ are positive integers. Then in the case where we have $2$ dice with size $m$, we have solutions with sizes $a$ and $ab^2$ whose generating functions are given by 
\[{x(x^a-1)\over x-1}\]
and \[{x(x^{m}-1)^2\over (x^a-1)(x-1)}\]
respectively.
\end{theorem}
\begin{proof}
Taking the product of these functions, we get
$$ \left(\frac {x(x^{a}-1)}{x-1}\right)\left(\frac {x(x^{m}-1)^2}{(x^a -1)(x-1)}\right) = \left(\frac {x(x^{m}-1)}{(x-1)}\right)^2$$
which is the frequency polynomial for two $m$ sided dice. Thus, we need only demonstrate that \[\left(\frac {x(x^{m}-1)^2}{(x^a -1)(x-1)}\right)\]
has non-negative coefficients. Notice that

\begin{align*}
\frac{x(x^{m}-1)^2}{(x^a -1)(x-1)}
&=\left({x(x^a-1)\over x-1}\right)\left({x^m-1 \over x^a-1}\right)^2\\
&=\left(\sum_{i=1}^a x^i\right)\left(\sum_{i=0}^{b-1} x^{ia}\right)^2
\end{align*}
which has positive coefficients. Observe evaluating this polynomial at $x=1$ gives $ab^2$ as desired. 
\end{proof}

Further computations allows us to explicitly describe the labels on the dice.
\begin{corollary}\label{cor:carlye}
Let $m=ab$ with $a$ and $b$ non-negative integers. 
Consider a dice size $ab^2$ who labels come from $\{1,2,\dots, 2m-a\}$ in the following way:
\begin{enumerate}
\item the numbers $(i-1)a+1,(i-1)a+2,\dots, ia, 2m-(i+1)a+1, 2m-(i+1)a+2,\dots, 2m-ai$ each appear $i$ times on the dice for $1\leq i\leq b-1$; and 
\item the numbers $m-a+1,m-a+2,\dots, m$ appear $b$ times on the dice.
\end{enumerate}

This dice, along with a standard $a$-sided dice, has the same frequencies of sums as two standard $m$-sided dice.
\end{corollary}

\begin{proof}
This follows from the prior Theorem provided the polynomial corresponding to \[
\frac{x(x^{m}-1)^2}{(x^a -1)(x-1)}\]
is the generating function for the dice of size $ab^2$. Continuing from the prior Theorem's proof, we have

\begin{align*}
\frac{x(x^{m}-1)^2}{(x^a -1)(x-1)}
&=\left(\sum_{i=1}^a x^i\right)\left(\sum_{i=0}^{b-1} x^{ia}\right)^2\\
&=x\left(\sum_{i=0}^{b-1} x^{ia}\right)^2+x^2\left(\sum_{i=0}^{b-1} x^{ia}\right)^2+\cdots +x^a\left(\sum_{i=0}^{b-1} x^{ia}\right)^2\\
\end{align*}
Let us denote $\displaystyle c_j(x):=x^j\left(\sum_{i=0}^{b-1} x^{ia}\right)^2$, and so the above is $\displaystyle  \sum_{j=1}^a c_j$. Observe that 
\[c_j(x)=x^j+2x^{j+a}+3x^{j+2a}+\cdots +bx^{j+m-a}+(b-1)x^{j+m}+\cdots +x^{j+2m-2a}.\]
Thus
\begin{align*}
\displaystyle  \sum_{j=1}^a c_j&=\sum_{i=1}^b\sum_{j=1}^a ix^{j+ai-a}+\sum_{i=1}^{b-1}\sum_{j=1}^a (b-i)x^{j+m+ia},
\end{align*}
which is the generating function for the aforementioned dice.
\end{proof}

\begin{remark}
At the end of \cite{gc}, the origin of Question \ref{qu:dif_size}, they mentioned as an example that one can have a four-sided dice with labels $1,1,4,4$ and a nine-sided dice labeled $1,2,2,3,3,3,4,4,4,5,5,5,6,6,6,7,7,8$. We point out that our construction could not generate this solution, since neither $4$ nor $9$ are divisors of $6$. Thus, there is much room for further exploration here.
\end{remark}

\begin{example}
We use the fact that $6=2\times 3$ for an immediate application of the prior result.
If we let $a=3$, and $b=2$, the prior results gives us two dice with sizes $3$ and $12$. These dice have the labels $1,2,3$ and $1,2,3,4,4,5,5,6,6,7,8,9$.

One can use the following frequency table to verify the sums of dice with the above sides yield the same sums of two standard $6$ sided dice.

\begin{center}

\begin{tabular}{l|llllllllllll}
  & $1$ & $2$ & $3$ & $4$ & $4$ & $5$ & $5$ & $6 $ & $6$ & $7$  & $8$  & $9$  \\ \hline
$1$ & $2$ & $3$ & $4$ & $5$ & $5$ & $6$ & $6$ & $7$ & $7$ & $8 $ &$ 9$  & $10$ \\
$2$ & $3$ & $4$ & $5$ & $6$ & $6$ & $7$ & $7$ & $8$ & $8$ & $9$  & $10$ & $11$ \\
$3$ & $4$ & $5$ & $6$ & $7$ & $7$ & $8$ & $8$ & $9$ & $9$ & $10$ & $11 $& $12$
\end{tabular}
\end{center}
\end{example}

This prior result surprisingly provides an alternative version of a known combinatorial identity. First, recall  $T_n$ is the number of boxes in a triangular array with $n$ rows and $i$ boxes in row $i$. See figure \ref{fig:triangles}.

\begin{figure}[h]
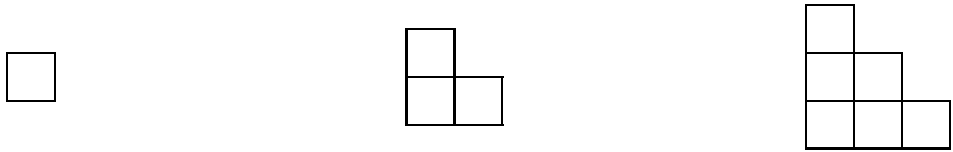

\begin{center}
\begin{minipage}{.32\textwidth}
\begin{center}

\begin{ytableau}
       {} &\none & \none &  \none \\
  \end{ytableau}
\end{center}
\end{minipage}
\begin{minipage}{.32\textwidth}

\begin{center}
\begin{ytableau}
       {} &\none & \none &  \none \\
  {} & {}& \none & \none  \\
  \end{ytableau}
\end{center}
\end{minipage}
\begin{minipage}{.32\textwidth}

\begin{center}
\begin{ytableau}
       {} &\none & \none &  \none \\
  {} & {}& \none & \none  \\
  {} & {}& {} & \none \\
  \end{ytableau}
\end{center}
\end{minipage}
\end{center}
\caption{The triangular arrays with 1, 2, and 3 rows. Thus, $T_1=1$, $T_2=3$, and $T_3=6$.}
\label{fig:triangles}
\end{figure}

The following is an established combinatorial identity, but we shall reprove it here.
\begin{proposition}
If $n\in \N$ then \[n^2=T_n+T_{n-1}.\]
\end{proposition}\label{prop:old_triangle}
\begin{proof}
Adding the formulas for $T_n$ and $T_{n-1}$ together, we have

\begin{tabular}{r c c c c c c c c c c c c}
$T_n$&$=$&$n$&$+$&$(n-1)$&$+$&$(n-2)$&$+$&$\cdots$& $+$&$1$\\
$+T_{n-1}$&$=$&&&$1$&$+$&$2$&$+$&$\cdots$& $+$&$(n-1)$\\ \hline
$T_n+T_{n-1}$&$=$&$n$&$+$&$n$&$+$&$n$&$+$&$\cdots$& $+$&$n$\\
&$=$&$n^2$
\end{tabular}

\end{proof}

A popular combinatorial way to interpret this identity is by witnessing that an $n\times n$ grid is made up of the diagrams of $T_n$ and $T_{n-1}$. See Figure \ref{fig:old_tts}

\begin{figure}[h]
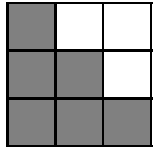

\begin{center}

\begin{ytableau}
  *(gray) & {}& {} & \none \\
  *(gray) & *(gray)& {} & \none \\
  *(gray) & *(gray)&*(gray) & \none \\
  \end{ytableau}
\end{center}
\caption{A visual proof that $T_2+T_3=3^2$}
\label{fig:old_tts}
\end{figure}


We now get the following Corollary to the prior result.
\begin{corollary}
If $a\in \N$ divides $m\in \N$, so that $m=ab$ for some integer $b$, we have 
\[m^2=a^2(T_b+T_{b-1})\]
\end{corollary}\label{cor:new_triangle}

\begin{proof}
Let $\displaystyle P(x):={x(x^{m}-1)^2\over (x^a-1)(x-1)}$. When written as a Taylor series, we see that $P(1)=ab^2={m^2\over a}$. On the other had, using the prior proof to express $P(x)$ as 
\[P(x)=\sum_{i=1}^b\sum_{j=1}^a ix^{j+ai}+\sum_{i=1}^{b-1}\sum_{j=1}^a (b-i)x^{j+m-a+ia},\] note that the coefficients of the first $ab$ terms sum to $aT_b$, while the remaining $a(b-1)$ terms sum to $aT_{b-1}$. Thus, we have demonstrated 
\[{m^2\over a}=aT_{b}+aT_{b-1},\]
and so 
\[m^2=a^2(T_{b}+T_{b-1}).\]
\end{proof}

\begin{remark}
This Corollary really is just a restatement of Proposition \ref{prop:old_triangle}. Indeed, rather than multiplying by $a$ in the last step, one could have divided by $a$ to get
\[b^2=T_b+T_{b-1}\]
which we already know is true. However, the new statement, as written, does has following interesting combinatorial interpretation: Since $m=ab$, one may split up a grid of $m\times m$ squares into $b\times b$ squares, of which there are $a^2$ in total. Each of these $b\times b$ squares has $b^2=T_b+T_{b-1}$ squares. See Figure \ref{fig:new_tts}.
\end{remark}

\begin{figure}[h]
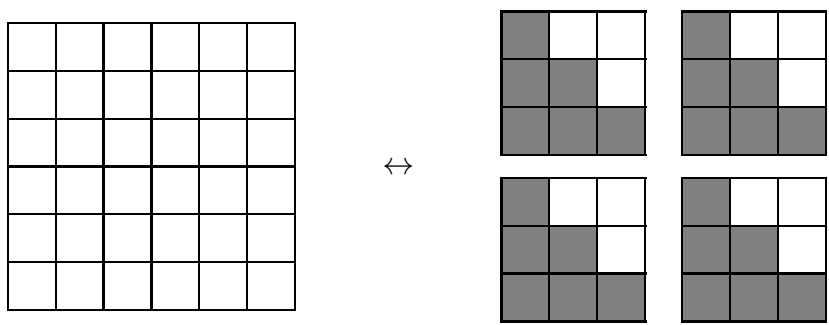

\begin{center}
\begin{minipage}{.3\textwidth}

\begin{ytableau}
  {} & {}& {} & {} & {}& {} & \none \\
  {} & {}& {} & {} & {}& {} & \none \\
  {} & {}& {} & {} & {}& {} & \none \\
  {} & {}& {} & {} & {}& {} & \none \\
  {} & {}& {} & {} & {}& {} & \none \\
  {} & {}& {} & {} & {}& {} & \none \\
  \end{ytableau}
\end{minipage}
$\leftrightarrow$\hspace{.4in}
\begin{minipage}{.1\textwidth}

\begin{ytableau}
  *(gray) & {}& {} & \none \\
  *(gray) & *(gray)& {} & \none \\
  *(gray) & *(gray)&*(gray) & \none \\
  \end{ytableau}\vspace{.1in}

\begin{ytableau}
  *(gray) & {}& {} & \none \\
  *(gray) & *(gray)& {} & \none \\
  *(gray) & *(gray)&*(gray) & \none \\
  \end{ytableau}
\end{minipage} \hspace{.2in}
\begin{minipage}{.1\textwidth}

\begin{ytableau}
  *(gray) & {}& {} & \none \\
  *(gray) & *(gray)& {} & \none \\
  *(gray) & *(gray)&*(gray) & \none \\
  \end{ytableau}\vspace{.1in}

\begin{ytableau}
  *(gray) & {}& {} & \none \\
  *(gray) & *(gray)& {} & \none \\
  *(gray) & *(gray)&*(gray) & \none \\
  \end{ytableau}
\end{minipage}
\end{center}
\caption{A visual proof that $4(T_{2}+T_{3})=6^2$. In this case, $a=2$ and $b=3$.}
\label{fig:new_tts}
\end{figure}

\section{Different solutions for dice of different size }\label{sec:dif_sizes}

In this section, we report on preliminary results of Question \ref{qu:dif_dice_size}.

\begin{proposition}\label{lem:prime}
    Let $p$ and $q$ be prime numbers. There is no way to relabel dice of size $p$ and $q$ without changing the frequencies of their sums. 
\end{proposition}

\begin{proof}
    The frequency polynomial for dice of size $p$ and $q$ is 
\begin{align*}
    F(x) &= \frac{x(x^{p}-1)}{x-1} \cdot \frac{x(x^{q}-1)}{x-1}\\
&= x^2 \cdot \phi_{p} \cdot \phi_{q}.
\end{align*}

  Since $\phi_{p}{(1)} = p$ and $\phi_{q}{(1)} = q$, the only factorization of $F(x)$ that gives solutions is the factorization into polynomials $x\phi_p$ and $x\phi_q$. 
    
\end{proof}

\begin{remark}
We originally conjectured that the prior result can be extended to two numbers which are relatively prime, but thanks to communication with David Rusin, we discovered this is false. The frequency polynomial for a $5$ and $6$ sided dice, for instance, can be factored into $x\phi_5\phi_6$ and $x\phi_2\phi_3$, and notably, both polynomials have positive coefficients.
\end{remark}

%
%
%
%


We have a preliminary result in the case where the numbers are not relatively prime. 
\begin{lemma}\label{lem:prime_power}
 There are $k$ ways to relabel a dice of size $p$ and a dice of size $p^k$, where $p$ is prime.
    
\end{lemma}

\begin{proof}
The frequency polynomial in this case is 
\[F(x)={x(x^p-1)\over x-1}{x(x^{p^k}-1)\over x-1}=x^2\phi_p^2\phi_{p^2}\phi_{p^3}\cdots \phi_{p^k}.\]

Recall that $\phi_{p^i}(1)=p$ for all $i\geq 1$ and has non-negative coefficients. Thus, there are $k$ different choices one can make for the dice of size $p$, which determines the dice of size $p^k$.
\end{proof}

\section{Revisiting dice with size $p^k$ }\label{sec:prime_power}

In this section, we wish to refine \cite[Theorem 10]{gc}. We restate this result and include its proof here.

\begin{theorem}
There are exactly $\displaystyle {2k-1\choose k-1}$ solution dice of size $p^k$ for all positive integers $k$ and all prime numbers $p$.
\end{theorem}
\begin{proof}
Consider a polynomial of the form 
\[P(x)=x\prod_{1\leq i\leq k}\phi_{p^i}(x)^{c_i}.\]
This polynomial is a solution so long as $P(1)=p^k$, and since $\phi_{p^i}(1)=p$ for all $i\geq 1$, our solutions are in bijection with non-negative integer solutions to 
\[ \sum_{i=1}^kc_i=k,\]
of which there are $\displaystyle {2k-1\choose k-1}$.
\end{proof}

As a follow up to this result, we ask the following.
\begin{question}
Given $n$ standard dice of size $p^k$, how many other solution dice are there?
\end{question}

The reason this refines the question is because it is not necessarily the case that each of the aforementioned ${2k-1\choose k-1}$ solutions would be a solution for all possible values of $n$. In fact, the only time we can guarantee that all of these are solution dice is if $n\geq k$, as the above Theorem assumes that we always have enough dice available for a dice to be considered a solution.\footnote{This is a notion also explored by \cite{gc}. They call the \textit{game size} of a dice to be the smallest number of dice needed for a given dice to be considered a solution. This notion was not brought up in the discussion surrounding \cite[Theorem 10]{gc}, leaving some room for further exploration.} Indeed, if $n<k$, not all proposed solutions may be able to be combined with $n-1$ other dice to give the same frequency distribution. This is demonstrated by the following Proposition, for the case where $n=2$.

\begin{proposition}\label{prop:josh}
  For 2 dice of size $p^k$, where $p$ is prime, there are 
  $$\sum^{\left\lfloor \frac{k}{2} \right\rfloor}_{i=0}{\binom{k}{i}\binom{k-i}{i}}$$ solutions.
\end{proposition}
\begin{proof}
  The frequency polynomial for 2 dice of size $p^k$ is 
\begin{align*}
\left(\frac{x(x^{p^k}-1)}{x-1}\right)^2 = x^2 \cdot \phi_{p}^2 \cdot \phi_{p^2}^2 \dots \phi_{p^k}^2.
\end{align*}  
  All solutions will have the form $x^{c_1} \cdot \phi_{p}^{c_2} \cdot \phi_{p^2}^{c_3} \cdots \phi_{p^k}^{c_{k+1}}$.
  We know $c_1 = 1$ since every solution must have a single factor of $x$ because each die is required to have that minimum on one of its sides. Any solution with generating function $P(x)$ satisfies $P(1) = p^k$ since $p^k$ is the size of the two dice which we are trying to find more solutions for. Thus, the solutions correspond to $c_2 + c_3 + \cdots+ c_{k+1} = k$, assuming that $c_i \leq 2$ since we only have two dice in this case. Thus, solutions are in bijection with tuples in $\{0,1,2\}^k$ whose components sum to $k$. Given $i\leq \lfloor {k\over 2}\rfloor$, there are ${k\choose i}$ ways of choosing where the $2$'s appears in this tuple. From there, we must have $k-2i$ positions with a $1$. There are ${k-i\choose k-2i}$ choices for where to place the $1$'s. The remaining spots must be $0$'s. Thus, we've shown the number of solutions is precisely the sum give in the statement.
%
%
\end{proof}

Table \ref{tab:sols} compares the number of solutions for the $2$ dice case versus those given in \cite[Theorem 10]{gc}.

\begin{table}[h]
\begin{tabular}{c | c | c}
$k$ & $\displaystyle \sum^{\left\lfloor \frac{k}{2} \right\rfloor}_{i=0}{\binom{k}{i}\binom{k-i}{i}}$ & $\displaystyle {2k-1\choose k-1}$\\ \hline

$1$ & $1$ & $1$ \\ \hline
$2$ & $3$ & $3$ \\ \hline
$3$ & $7$ & $10$ \\ \hline
$4$ & $19 $ & $35$ \\ \hline
$5$ & $51$ & $126$

\end{tabular}
\caption{Comparing the number of actual solutions for two dice with size $p^k$ with \cite[Theorem 10]{gc}. Observe the values disagree if and only if $k>2$.}\label{tab:sols}
\end{table}

We can generalize this beyond the case where $n=2$, though we do not necessarily get a concise formula as we did above.

\begin{proposition}
Suppose we have $n$ dice of size $p^k$. If $n<k$, then the number of solutions is precisely
\[[x^k](1+x+x^2+\cdots+x^n)^k.\]
\end{proposition}
\begin{proof}
As before, we consider possible choices of $c_1,\dots, c_k$ so that \[x\prod_{i=1}^k \left(\phi_{p^i}(x)\right)^{c_i}\] is a solution. Recall that $\phi_{p^i}(1)=p$ for all $i\geq 1$, so the solutions are in bijections with non-negative solutions to \[c_1+c_2+\cdots+c_k=k\]
where $c_i\leq n$. Such solutions are enumerated by the statement of this proposition.
\end{proof}

This provides an alternative proof of Proposition \ref{prop:josh}.
\begin{corollary}
The number of solutions for $2$ dice of size $p^k$ is 
\[\sum_{i=0}^{\left \lfloor {k\over 2}\right\rfloor}{k\choose i}{k-i\choose i}.\]
\end{corollary}
\begin{proof}
Using the Binomial Theorem, we have
\begin{align*}
(1+x+x^2)^k&=\sum_{i=0}^k{k\choose i}(1+x)^{k-i}x^{2i}\\
&=\sum_{i=0}^k{k\choose i}x^{2i}\sum_{j=0}^{k-i}{k-i\choose j} x^{k-i-j}\\
&=\sum_{i=0}^k\sum_{j=0}^i{k\choose i}{k-i\choose j} x^{k+i-j}
\end{align*}

Thus, the coefficient of $[x^k]$ for this polynomial is found when $i=j$, yielding
\[\sum_{i=0}^k{k\choose i}{k-i\choose i}.\]
Note the terms with $i>\left\lfloor{k\over 2}\right\rfloor$
are all $0$ due to the second binomial.
\end{proof}

\begin{remark}
The sum given in the prior Corollary is referred to as the ``central trinomial coefficient" since is the middle term of $(1+x+x^2)^k$.
\end{remark}

\appendix

\section{Rational Function Simplifications}

In this section, we provide work to help justify the simplification of certain products of cyclotomic polynomials given in sections \ref{sec:p2q} and \ref{sec:pqr}. 

\subsection{Computations for Section \ref{sec:p2q}}\label{sec:comp_p2q}

\begin{align*}
A_{1102}&=\phi_q\phi_p\phi_{p^2}\phi_{p^2q}^2\\
&={1-x^q\over 1-x}{1-x^p\over 1-x}{1-x^{p^2}\over 1-x^p}{(1-x^p)^2(1-x^{p^2q})^2\over (1-x^{p^2})^2(1-x^{pq})^2}\\
&=(1-x^q)(1-x^p)^2(1-x^{p^2q})^2\left(\displaystyle \sum_{i=0}^\infty  x^i\right)^2\left(\displaystyle \sum_{i=0}^\infty  x^{ip^2}\right)\left(\displaystyle \sum_{i=0}^\infty  x^{ipq}\right)^2\\ \\ \\ \\
A_{2022}&=\phi_q\phi_p^2\phi_{pq}^2\phi_{p^2q}^2\\
&={1-x^q\over 1-x}{(1-x^p)^2\over (1-x)^2}{(1-x^{pq})^2(1-x)^2\over (1-x^p)^2(1-x^q)^2}{(1-x^p)^2(1-x^{p^2q})^2\over (1-x^{p^2})^2(1-x^{pq})^2}\\
&=(1-x^p)^2(1-x^{p^2q})^2\left(\displaystyle \sum_{i=0}^\infty  x^i\right)\left(\displaystyle \sum_{i=0}^\infty  x^{iq}\right)\left(\displaystyle \sum_{i=0}^\infty  x^{ip^2}\right)^2\\ \\ \\ \\
A_{2012}&=\phi_q\phi_p^2\phi_{pq}\phi_{p^2q}^2\\
&={1-x^q\over 1-x}{(1-x^p)^2\over (1-x)^2}{(1-x^{pq})(1-x)\over (1-x^p)(1-x^q)}{(1-x^p)^2(1-x^{p^2q})^2\over (1-x^{p^2})^2(1-x^{pq})^2}\\
&=(1-x^p)^3(1-x^{p^2q})^2\left(\displaystyle \sum_{i=0}^\infty  x^i\right)^2\left(\displaystyle \sum_{i=0}^\infty  x^{ip^2}\right)^2\left(\displaystyle \sum_{i=0}^\infty  x^{ipq}\right)\\ \\
\end{align*}

\begin{align*}
\end{align*}

\subsection{Computations for Section \ref{sec:pqr}}\label{sec:comp_pqr}
{\footnotesize
\begin{align*}
A_{0222}&=\phi_p\phi_q\phi_r\phi_{pr}^2\phi_{qr}^2\phi_{pqr}^2\\
&={1-x^p\over 1-x}{1-x^q\over 1-x}{1-x^r\over 1-x}{(1-x^{pr})^2(1-x)^2\over (1-x^p)^2(1-x^r)^2}{(1-x^{qr})^2(1-x)^2\over (1-x^q)^2(1-x^r)^2}{(1-x^{pqr})^2(1-x^p)^2(1-x^q)^2(1-x^r)^2 \over (1-x^{pq})^2(1-x^{pr})^2(1-x^{qr})^2(1-x)^2}\\
&=(1-x^p)(1-x^q)(1-x^{pqr})^2\left(\displaystyle \sum_{i=0}^\infty  x^{i}\right)\left(\displaystyle \sum_{i=0}^\infty  x^{ir}\right)\left(\displaystyle \sum_{i=0}^\infty  x^{ipq}\right)^2\\ \\  \\ \\
A_{0122}&=\phi_p\phi_q\phi_r\phi_{pr}\phi_{qr}^2\phi_{pqr}^2\\
&={1-x^p\over 1-x}{1-x^q\over 1-x}{1-x^r\over 1-x}{(1-x^{pr})(1-x)\over (1-x^p)(1-x^r)}{(1-x^{qr})^2(1-x)^2\over (1-x^q)^2(1-x^r)^2}{(1-x^{pqr})^2(1-x^p)^2(1-x^q)^2(1-x^r)^2 \over (1-x^{pq})^2(1-x^{pr})^2(1-x^{qr})^2(1-x)^2}\\
&=(1-x^p)^2(1-x^q)(1-x^{pqr})^2\left(\displaystyle \sum_{i=0}^\infty  x^{i}\right)^2\left(\displaystyle \sum_{i=0}^\infty  x^{ipr}\right)\left(\displaystyle \sum_{i=0}^\infty  x^{ipq}\right)^2\\ \\ \\ \\
A_{2001}&=\phi_p\phi_q\phi_r\phi_{pq}^2\phi_{pqr}\\
&={1-x^p\over 1-x}{1-x^q\over 1-x}{1-x^r\over 1-x}{(1-x^{pq})^2(1-x)^2\over (1-x^p)^2(1-x^q)^2}{(1-x^{pqr})(1-x^p)(1-x^q)(1-x^r) \over (1-x^{pq})(1-x^{pr})(1-x^{qr})(1-x)}\\
&=(1-x^r)^2(1-x^{pq})(1-x^{pqr})\left(\displaystyle \sum_{i=0}^\infty  x^{i}\right)^2\left(\displaystyle \sum_{i=0}^\infty  x^{ipr}\right)\left(\displaystyle \sum_{i=0}^\infty  x^{iqr}\right)\\ \\ \\ \\
A_{1112}&=\phi_p\phi_q\phi_r\phi_{pq}\phi_{pr}\phi_{qr}\phi_{pqr}^2\\
&={1-x^p\over 1-x}{1-x^q\over 1-x}{1-x^r\over 1-x}{(1-x^{pq})(1-x)\over (1-x^p)(1-x^q)}{(1-x^{pr})(1-x)\over (1-x^p)(1-x^r)}{(1-x^{qr})(1-x)\over (1-x^q)(1-x^r)}{(1-x^{pqr})^2(1-x^p)^2(1-x^q)^2(1-x^r)^2 \over (1-x^{pq})^2(1-x^{pr})^2(1-x^{qr})^2(1-x)^2}\\
&=(1-x^p)(1-x^q)(1-x^r)(1-x^{pqr})^2\left(\displaystyle \sum_{i=0}^\infty  x^{i}\right)^2\left(\displaystyle \sum_{i=0}^\infty  x^{ipq}\right)\left(\displaystyle \sum_{i=0}^\infty  x^{ipr}\right)\left(\displaystyle \sum_{i=0}^\infty  x^{iqr}\right)
\end{align*}
}

\bibliographystyle{plain} 
\bibliography{ref[1]}


\end{document}